\documentclass[11pt]{amsart} \textwidth=14.5cm \oddsidemargin=1cm
\usepackage{amsmath,amsthm,amssymb,latexsym,epic,bbm,comment,yfonts,mathrsfs}
\usepackage{graphicx,enumerate,stmaryrd,rotating}
\usepackage[all]{xy}
\xyoption{2cell}

\allowdisplaybreaks

\usepackage[active]{srcltx}
\usepackage[parfill]{parskip}

\usepackage{youngtab}

\usepackage[latin1]{inputenc}
\usepackage{amsxtra}
\usepackage{amsmath}
\usepackage{amscd}
\usepackage{amssymb}
\usepackage{amsfonts}
\usepackage[all]{xy}
\usepackage{mathrsfs}
\usepackage{amsthm}
\usepackage{color}
\usepackage{upgreek}
\usepackage{verbatim}  %%for \begin{comment}
\usepackage{enumerate}
\usepackage{latexsym}
\usepackage{graphicx}
\usepackage{tikz}
\usepackage{todonotes}
\usepackage{setspace}
\usepackage{hyperref}
\usepackage{stmaryrd}
\usepackage{nicefrac}
\usepackage{euscript}
\usetikzlibrary{decorations.pathmorphing}

 \definecolor{darkgreen}{HTML}{336633}
 \definecolor{darkred}{HTML}{993333}
\makeatletter

\newcommand{\arxiv}[1]{\href{http://arxiv.org/abs/#1}{\tt
    arXiv:\nolinkurl{#1}}}

\theoremstyle{plain}
\newtheorem{thm}{Theorem}%[section]
\newtheorem*{thm*}{Theorem}
\newtheorem*{thmA}{Theorem A}
\newtheorem*{thmB}{Theorem B}
\newtheorem*{thmC}{Theorem C}

\newtheorem{lem}[thm]{Lemma}
\newtheorem{prop}[thm]{Proposition}

\newtheorem{cor}[thm]{Corollary}
\newtheorem{df-prop}[thm]{Definition-Proposition}

\theoremstyle{definition}

\theoremstyle{remark}
\newtheorem{rem}[thm]{Remark}
\newtheorem{ex}[thm]{Example}

\def\onto{\twoheadrightarrow}

\def\mod{\operatorname{-mod}\nolimits}

\def\Hom{\operatorname{Hom}\nolimits}

\def\Res{\operatorname{Res}\nolimits}
\def\Ind{\operatorname{Ind}\nolimits}

\def\gl{\mathfrak{gl}}

\def\la{\lambda}

\def\pn{\mf{pe} (n)}
\def\ov{\overline}

\newcommand{\mc}{\mathcal}
\newcommand{\mf}{\mathfrak}
\newcommand{\C}{\mathbb C}
\newcommand{\oo}{{\ov 0}}
\newcommand{\oa}{{\bar 0}}
\newcommand{\ob}{{\bar 1}}
\newcommand{\vare}{\epsilon} %%%% change-original \vere=\varepsilon

\newcommand{\ad}{\mathrm{ad}}
\newcommand{\cP}{\mathcal{P}}
\newcommand{\cF}{\mathcal{F}}

\newcommand{\fg}{\mathfrak{g}}

\newcommand{\fb}{\mathfrak{b}}

\newcommand{\fh}{\mathfrak{h}}
\newcommand{\fn}{\mathfrak{n}}
\newcommand{\n}{\mathfrak{n}}
\newcommand{\mZ}{\mathbb{Z}}
\newcommand{\cO}{\mathcal{O}}

\newcommand{\mC}{\mathbb{C}}

\newcommand{\h}{\mathfrak{h}}

\newcommand{\Coind}{{\rm Coind}}

\newcommand{\g}{\mathfrak{g}}
\newcommand{\fl}{\mathfrak{l}}
\newcommand{\fp}{\mathfrak{p}}
\newcommand{\fu}{\mathfrak{u}}
\newcommand{\Real}{\mathrm{Re}}

\newcommand{\cL}{\mathcal{L}}

 \def\tN{{\widetilde{\mc N}}}

  \newcommand{\add}{{\mathrm{add}}}

  \newcommand{\WG}{{\widetilde{\Gamma}_\zeta}}
  \newcommand{\Whoa}{{\text{Wh}_\zeta}}
   \newcommand{\Whob}{{\widetilde{\text{Wh}}_\zeta}}

\begin{document}

\numberwithin{equation}{section}

\title[Whittaker modules for classical Lie superalgebras]{Whittaker modules for classical Lie superalgebras}

\author{Chih-Whi Chen}
\date{}

\begin{abstract}   
	We classify simple Whittaker modules for classical Lie superalgebras in terms of their parabolic decompositions.	We establish a type of  Mili{\v{c}}i{\'c}-Soergel  equivalence of a category of Whittaker modules and a category of Harish-Chandra bimodules. For classical Lie superalgebras of type I, we reduce the problem of composition factors of standard Whittaker modules to that of Verma modules in their BGG categories $\mc O$. As a consequence, the composition series of standard Whittaker modules over the general linear Lie superalgebras $\gl(m|n)$ and the ortho-symplectic Lie superalgebras $\mf{osp}(2|2n)$ can be computed via the  Kazhdan-Lusztig combinatorics. % We show that the composition factors of standard Whittaker modules over the general linear Lie superalgebras $\gl(m|n)$ and the ortho-symplectic Lie superalgebras $\mf{osp}(2|2n)$ are computed by Kazhdan-Lusztig combinatorics. For the periplectic Lie superalgebras $\pn$, we reduce this problem to the problem of irreducible characters of its category $\mc O$.  %We then obtain equivalences of blocks of Whittaker modules. 

	% We reduce the problem of composition factors of standard Whittaker modules over classical Lie superalgebras to that of Verma modules in its BGG category $\mc O$. Therefore for standard Whittaker modules over the general linear Lie superalgebras $\gl(m|n)$ and the ortho-symplectic Lie superalgebras $\mf{osp}(2|2n)$ are computed by Kazhdan-Lusztig combinatorics

	%over the general linear Lie superalgebras $\gl(m|n)$ and the ortho-symplectic Lie superalgebras $\mf{osp}(2|2n)$ are computed by Kazhdan-Lusztig combinatorics. For the periplectic Lie superalgebras $\pn$, we reduce this problem to the problem of irreducible characters of its category $\mc O$. 	We establish an  Mili{\v{c}}i{\'c}-Soergel type equivalence of blocks of Whittaker categories and Harish-Chandra bimodules. %We then obtain equivalences of blocks of Whittaker modules. 
		
\end{abstract}

\maketitle

\tableofcontents 


\vspace{5mm}

\section{Introduction}\label{sec1}

\subsection{}
In the classical paper \cite{Ko78}, Kostant introduced and classified a family of simple modules $Y_{\xi,\eta}$ over finite-dimensional complex semisimple Lie algebras. Motivated by the study of Whittaker models, he found the condition of the existence of a Whittaker vector for  a simple module of linear  semisimple Lie group.  Subsequently, a systematic construction of the {\em Whittaker modules} in the category $\mc N$ over  finite-dimensional complex semisimple Lie algebras, containing Kostant's  simple modules and modules in the BGG category $\mc O$,  was studied by McDowell in \cite{Mc,Mc2} and by Mili{\v{c}}i{\'c} and Soergel  in \cite{MS,MS2}. 

It is known in \cite{MS} that $\mc N$ has certain {\em standard Whittaker modules} parametrized by cosets in the Weyl group of a certain subgroup (see also \cite{Mc}). In particular, an equivalence of certain categories of Whittaker modules and Harish-Chandra bimodules was established by  Mili{\v{c}}i{\'c} and Soergel in \cite[Theorem 5.1]{MS}. As an application, the problem of composition factors of the standard Whittaker modules was partially solved in \cite[Section 5]{MS}. Namely, this solution follows from the composition factors  of Verma modules in the BGG category $\mc O$, which is reduced to the Kazhdan-Lusztig conjectures (see, e.g., \cite{BB, BK,KL1}). Around the same time, Backelin  in \cite{B}  developed a complete solution to the same problem using {\em Whittaker functors}. %of composition factors in standard Whittaker modules over  finite-dimensional complex semisimple Lie algebras.
  Consequently, this problem can be completely calculated by the Kazhdan-Lusztig combinatorics (see \cite[Theorem 6.2]{B}). 
 %, which by now are a theorem .

   There have been numerous attempts to  obtain results toward the study of Whittaker modules for Lie algebras and related algebras that possess a structure similar to triangular decomposition; see, e.g., \cite{ALZ,Be,Chri,CDH,GLZ1,LZ1,LWZ,O1,OW,Sev,Wa2} and references therein. Inspired by these activities, Batra and Mazorchuk developed in \cite{BM} a general framework for  Whittaker modules. More recently, Coulembier and Mazorchuk studied in \cite{CoM} the extension fullness of the {\em Whittaker categories}.

 \subsection{} While there are now complete solutions to the problem of composition factors in standard Whittaker modules for the semisimple Lie algebras, the Whittaker modules for Lie superalgebras were not investigated until recently. In a recent exposition \cite{BCW}, Bagci, Christodoulopoulou and Wiesner  initiated the study of Whittaker modules over Lie superalgebras in a systematic fashion, where some simple and standard Whittaker modules over basic classical Lie superalgebras of type I were also constructed. In particular, these standard  Whittaker modules  are of finite length. Therefore it is natural to classify simple modules and study the composition series of modules in the category of Whittaker modules. There have also been a variety of work done to study Whittaker modules and W-algebras over basic Lie superalgebras; see, e.g., \cite{BG2,Xi,ZS}.
 
  We study several aspects of Whittaker modules over classical Lie superalgebras. Namely, the present paper attempts to classify simple Whittaker modules  and to construct a type of Mili{\v{c}}i{\'c}-Soergel equivalence between a category of Whittaker modules and a corresponding category of Harish-Chandra bimodules and to compute composition series of standard Whittaker modules. %answer these questions as follows.

\subsection{} Recall that a finite-dimensional Lie superalgebra $\fg=\fg_{\oa}\oplus\fg_{\ob}$ is called  {\em classical} if the restriction of the adjoint representation of $\fg$ to the Lie algebra $\fg_{\oa}$ is semisimple.   The Killing-Cartan type classification of finite-dimensional complex simple Lie superalgebras has been established by Kac in his celebrated paper \cite{Ka1,Ka2}. One of the most interesting subclass of Kac's list is the following series of classical Lie superalgebras:
\begin{align} 
&\gl(m|n),~\mf{sl}(m|n),~\mf{psl}(n|n),~\mf{osp}(m|2n),~D(2,1|\alpha),G(3),F(4),\label{Kaclist1} \\ 
&\mf{p}(n),~[\mf{p}(n),\mf{p}(n)],~\mf{q}(n),~\mf{sq}(n),~\mf{pq}(n),~\mf{psq}(n).\label{Kaclist2}
\end{align}  We refer to \cite{ChWa12,Mu12} for more details of these Lie superalgebras.

A classical Lie superalgebra $\g$ is called {\em type I}, if it has  a compatible $\mathbb Z$-gradation of the form $\g= \g_{-1}\oplus \mf g_0 \oplus \g_1,$ with $\g_{\bar 0}=\g_0$, $\g_{\bar 1}=\mf g_{-1}\oplus \g_1$. In particular, we have the following Lie superalgebras of type I from the list \eqref{Kaclist1}-\eqref{Kaclist2}:
\begin{align} 
&(\text{Type } {\bf A}):~\gl(m|n),~\mf{sl}(m|n),~\mf{psl}(n|n), \label{eq::claA}  \\
&(\text{Type } {\bf C}):~\mf{osp}(2|2n),\label{eq::claC} \\
&(\text{Type } {\bf P}):~\mf{p}(n),~[\mf{p}(n),\mf{p}(n)]. \label{eq::claP}
\end{align}

\subsection{}  To explain the contents of the present paper in more detail, we start by explaining our precise setup. Let $\mf g$ be a classical Lie superalgebra. Now we  fix a Cartan subalgebra $\fh_{\oa}$ of $\fg_{\oa}$. Then we have a weight space decomposition  with the set $\Phi\subset\fh_{\oa}^\ast$ of  roots
$\fg\;=\;\bigoplus_{\alpha\in\Phi\cup  \{0\}}\fg^\alpha,$ where  $\fg^\alpha=\{X\in\fg\,|\, [h,X]=\alpha(h)X,\,\mbox{ for any $h\in\fh_{\oa}$}\}$. The  subalgebra $\fh:=\fg^0$ is usually referred to  as the {\em Cartan subalgebra} of $\g$; see also \cite[Section 1.3]{CCC}. 

%For any $\fh_{\oa}$-submodule $V$ of $\fg$, we write $\Phi(V)$ for the subset of $\Phi$ of weights appearing in $V$. In particular, we write $\Phi_{\oa}=\Phi(\fg_{\oa})$ and $\Phi_{\ob}=\Phi(\fg_{\ob})$.

We recall the notion of {\em parabolic decomposition} of classical Lie superalgebras from \cite[\S 2.4]{Ma} (see also  \cite{CCC,DMP,Mu12} and  references therein). For given $H\in\fh_{\oa}$ we can define subalgebras of $\fg$
\begin{equation}\label{deflu}\fl:=\bigoplus_{\Real \alpha(H)=0} \fg^\alpha,\quad \fu:=\bigoplus_{\Real \alpha(H)>0} \fg^\alpha, \quad \fu^-:=\bigoplus_{\Real \alpha(H)<0} \fg^\alpha,\end{equation} where $\Real(z)$ denotes the real part of $z\in \mathbb C$.
  Such a decomposition $\g=\fu^-\oplus\fl\oplus\fu$ gives rise to a  corresponding  {\em parabolic subalgebra}  $\mf p =\mf l\oplus \mf u$. We refer to $\mf l$ as the  {\em Levi subalgebra} of $\g$.  If $\mf l =\mf g^0= \mf h$, then \eqref{deflu} leads to a {\em triangular decomposition}. In this case, we write
 $\mf n:=\mf u,~\mf n^-:=\mf u^-$ and define the {\em Borel subalgebra} $\mf b:=\fh\oplus\fn$ (see also \cite[Section 3.3]{Mu12}). 
  
  Throughout the present paper, we fix a triangular decomposition $\mf g=\mf n^-\oplus \mf h \oplus \mf n$ with Borel subalgebra $\mf b=\mf h\oplus \mf n$. Then it gives a triangular decomposition 
  \begin{align}
  &\mf g_\oa=\mf n^-_\oa\oplus \mf h_\oa \oplus \mf n_\oa \label{eqtraigoa}
  \end{align} of $\mf g_\oa$ as well. Unless mention otherwise,  we choose the vector $H$ such that  $\mf p\supseteq \mf b$, $\mf l=\mf l_\oa$ and $\alpha(H)\in \mathbb R$, for all $\alpha \in \Phi$. We may note that $\mf h=\mf h_\oa$. %, in a given parabolic decomposition of $\g$. We also assume  that   the Levi subalgebra $\fl$ is purely even in our parabolic decompositions, that is, $\fl=\fl_\oa$. 
  See also \cite[Section 1.3]{CCC} for the definition of {\em reduced} parabolic   subalgebras.  %$\fl(H)=\fg^0=:\fh$ and $\fh=\fh_\oa$.

\subsection{} We denote by $\mf g$-Mod and $\mf g_\oa$-Mod the category of all $\mf g$-modules and $\mf g_\oa$-modules, respectively. There is a category $\mc W(\g,\mf n)$ consisting of $\g$-modules that are locally finite over $\mf n$ as considered in \cite{BCW}. To study simple modules in $\mc W(\g,\mf n)$, we propose a full subcategory  $\widetilde{\mc N}$  of $\mc W(\g,\mf n)$, which contains modules in the BGG category $\mc O$.  That is, the category  $\widetilde{\mc N}$  consists of finitely-generated $\g$-modules that are locally finite over $\mf n$ and over the center $Z(\mf g_\oa)$ of the universal enveloping algebra $U(\mf g_\oa)$. This is thus the category of $\g$-modules that are restricted to $\g_\oa$-modules in the category $\mc N$ of \cite{MS}. %We remark that our $\widetilde{\mc N}$ is a full subcategory of $\mc W(\mf g,\mf n)$.
We will prove that $\widetilde{\mc N}$ and $\mc W(\g, \mf n)$ have the same collection of simple objects; see Proposition \ref{lem::simples} and Remark \ref{Rmk2}. In the present paper, the modules in $\widetilde{\mc N}$ are called  {\em Whittaker modules}.

We set $\mc I:=\{\zeta\in \mf n^\ast|~\zeta([\mf n_\oa,\mf n_\oa])=0, ~\zeta(\mf n_\ob)=0\}.$ Denote by $\Phi(\mf n_\oa)$ the set of roots in $\mf n_\oa$. 
For any $\zeta\in \mc I$, we define a subset of simple roots of $\mf g_\oa$:	
\begin{align}
&\Phi_\zeta:=\{\alpha \in \Phi(\mf n_\oa)|\zeta(\mf g_\oa^\alpha)\neq 0\}.%~\Pi_\zeta:=\{\alpha \in \Pi|\zeta(\mf g_\oa^\alpha)\neq 0\}.
\end{align} %We denote by $\Pi_\zeta$ the set of simple roots of $\mf g_\oa$ that lie in $\Phi_\zeta$. 
This gives rise to a parabolic decomposition  of $\mf g_\oa$ with corresponding %parabolic subalgebra $\mf p^f_\oa$ and 
Levi subalgebra $\mf l_\zeta$ generated by $\mf g_\oa^\alpha$'s ($\alpha\in \Phi_\zeta$) and $\mf h$.  %:=\mf h \oplus \bigoplus_{ \alpha \in \Phi_\zeta} \mf g^\alpha_\oa$. 

Let $\widetilde{\mc N}(\zeta)$ be the full subcategory of $\widetilde{\mc N}$ consisting of modules $M \in \tN$  such that $x-\zeta(x)$  acts locally nilpotently on $M$, for any $ x\in \mf n_\oa$. In the case when $\mf g=\mf g_\oa$ we write $\mc N(\zeta)$ instead, which has been considered in \cite[Section 1]{MS}. 
%\mc N(f):=\{N \in \mc N |~x-f(x) \text{ acts locally nilpotently on }N, \text{for any } x\in \mf n_\oa\}. 
We then have a decomposition $\widetilde{\mc N} =\bigoplus_{\zeta\in \mc I} \widetilde{\mc N}(\zeta)$ by  \cite[Theorem 3.2]{BCW}. 
%Our first main result is the following, which generalizes results in \cite[Section 4]{BCW} where the case of some basic classical Lie superalgebras of type I was considered.

Following notations in \cite{MS}, we denote the family of Kostant's simple modules $Y_{\xi,\eta}$ by $Y_\zeta(\la, \zeta),$ where $\la \in \h^\ast$ and $\zeta$ is a character on $\mf n_\oa$. In  \cite[Section 4]{BCW}, some interesting partial results concerning the simple Whittaker modules of $\widetilde{\mc N}$ for basic classical Lie superalgebras of type I were obtained, which we complete in Theorem \ref{mainthm1typeI} and the following theorem:
\begin{thmA} Let $\mf g$ be an arbitrary classical Lie superalgebra. Suppose that $\mf l_\zeta$ is a Levi subalgebra in a parabolic decomposition $\mf g=  {\mf u}_\zeta^- \oplus \mf l_\zeta \oplus  {\mf u}_\zeta$ of $\mf g$. Then, for any $\la \in \h^\ast$, the $\mf g$-module $\widetilde{M}(\la, \zeta)$ that is parabolically induced from $Y_\zeta(\la, \zeta)$ has simple top, denoted $\widetilde{L}(\la, \zeta)$, and the correspondence gives rise to a bijection between the sets of isomorphism classes of simple $\mf l_\zeta$-modules of the form $Y_\zeta(\la, \zeta)$ and simple $\mf g$-modules of $\widetilde{\mc N}(\zeta)$. 
\end{thmA}

 The modules  $\widetilde{M}(\la, \zeta)$ may be called {\em standard Whittaker modules} by analogy with the Whittaker modules of Lie algebras (see, e.g., \cite[Section 1]{MS}).   Suppose that $\mf g$ is a classical Lie superalgebra of type I. It was established in \cite[Theorem A]{CM} that the Kac induction functor $K(-):=\Ind^\g_{\mf g_0\oplus \mf g_{1}}(-)$ gives rise to a bijection between simple $\mf g$-modules  and simple $\mf g_\oa$-modules. Naturally, it turns out that the functor $K(-)$ provides a bijection between simple objects in $\mc N(\zeta)$ and in  $\widetilde{\mc N}(\zeta)$ without assuming that $\mf l_\zeta$ is a Levi subalgebra, leading to a classification of simple objects of the category $\mc W(\g,\mf n)$ from \cite{BCW} too; see Theorem \ref{mainthm1typeI}.

\subsection{} 
 In \cite{MS}, a powerful approach was developed to solve the problem of composition factors of standard Whittaker modules.  Before giving the results, we recall several preparatory definitions and notions from \cite[Section 5]{MS}. Let $\mf g$ be a reductive Lie algebra with $\zeta \in \mc I$. Then $\mf l_\zeta$ is a Levi subalgebra of $\mf g$. Any module $M\in {\mc N}(\zeta)$ decomposes into generalized eigenspaces $M_{\chi^{\mf l}_\mu}$ according to the central characters $\chi^{\mf l}_\mu$ of $\mf l:=\mf l_\zeta$ associated with the weight $\mu\in \h^\ast$. Let $\Upsilon\subset\fh^\ast$ denote the set of integral weights, that is, weights appearing in finite-dimensional $\mf g$-modules.   % For a given $\mu \in \h^\ast$, we set $\chi_{\mu}^{\mf l}: \mc Z(\mf l_f)\rightarrow \mathbb C$ to be the central character for $\mf l_f$ associated with   $\la$. 
	For a given weight $\la\in \h^\ast$, define $\Lambda:= \la +\Upsilon$ and 
	\begin{align}
	&{\mc N}(\Lambda, \zeta) = \{M\in {\mc N}(\zeta)|~M_{\chi^{\mf l}} = 0 \text{ unless }\chi^{\mf l} = \chi^{\mf l}_\mu, \text{ for some }\mu \in \Lambda\}.
	\end{align} 
	Then  ${\mc N}(\zeta)$ decomposes into blocks ${\mc N}(\Lambda, \zeta)$, for cosets  $\Lambda \in \h^\ast /\Upsilon$.

The study of Harish-Chandra bimodules for semisimple Lie algebras goes back at least to the work of Bernstein and Gelfand \cite{BG}. They established an equivalence of the category $\mc O$ and a category of  Harish-Chandra bimodules; see \cite[Theorem 5.9]{BG}.  Mili{\v{c}}i{\'c} and Soergel established in \cite[Theorem 5.1]{MS} an equivalence of $\mc N(\Lambda,\zeta)$ and a corresponding category of Harish-Chandra bimodules. An analogue of  Mili{\v{c}}i{\'c}-Soergel equivalence was also given and  used in \cite{KM}.

%\blue{Q1: The category of Harish-Chandra bimodules $\mc B(I_\la)$ has been studied in \cite[Section 3]{CC} for regular $\la$. It is natural to describe the category $\mc B(I_\la)$ for singular $\la$.}

 %Suppose that $\mf g$ is a classical Lie superalgebra such that $\mf l_\zeta$ is a Levi subalgebra in a parabolic decomposition $\mf g= {\mf u}_\zeta^- \oplus \mf l_\zeta \oplus {\mf u}_\zeta$ of $\mf g$. We will generalize results in \cite[Section 5]{MS} to that for our setup of Lie superalgebras.
 
 %Every module $M\in \widetilde{\mc N}(\zeta)$ decomposes into generalized eigenspaces $M_{\chi^{\mf l}_\mu}$ according to the central characters $\chi^{\mf l}_\mu$ of $\mf l:=\mf l_\zeta$ associated with the weight $\mu\in \h^\ast$, see Section \ref{sect412}. Let $\Upsilon\subset\fh^\ast$ denote the set of integral weights, that is, weights appearing in finite dimensional $\mf g$-modules.   % For a given $\mu \in \h^\ast$, we set $\chi_{\mu}^{\mf l}: \mc Z(\mf l_f)\rightarrow \mathbb C$ to be the central character for $\mf l_f$ associated with   $\la$. For any coset $\Lambda:= \la +\Upsilon$ in $ \mf h^\ast/\Upsilon$ we set
 %\begin{align} &\widetilde{\mc N}(\Lambda, \zeta) = \{M\in \widetilde{\mc N}(\zeta)|~M_{\chi^{\mf l}_\nu} = 0 \text{ unless }\chi^{\mf l}_\nu = \chi^{\mf l}_\mu, \text{ for some }\mu \in \Lambda\}. \end{align}  Then  $\widetilde{\mc N}(\zeta)$ decomposes into blocks $\widetilde{\mc N}(\Lambda, \zeta)$, for $\Lambda \in \h^\ast /\Upsilon$. 

The category $\mc O$ can be defined in a natural way in the setup of Lie superalgebras; see \cite{ChWa12,Mu12}. In particular, we have $\mc O \subset \widetilde{\mc N}(0)$. In  \cite{MaMe12},   Mazorchuk and Miemietz established an analogous equivalence of  the category $\mc O$ and  a category of  Harish-Chandra $(\mf g, \mf g)$-bimodules for Lie superalgebra $\g$. More recently, a similar version of equivalence of  the category  $\mc O$ and a category of Harish-Chandra $(\mf g, \mf g_\oa)$-bimodules has been established in \cite[Theorem 3.1]{CC}, which turns out to be a powerful tool in the study of primitive spectrum of the periplectic Lie superalgebras $\pn$. With this equivalence,  we will generalize results of \cite[Section 5]{MS} to that for our setup of Lie superalgebras in the present paper. 

Consider a classical Lie superalgebra $\mf g$ with $\zeta \in \mc I$. We define the Weyl group $W$ of $\g$  to be the Weyl group of $\g_\oa$. % such that $\mf l_\zeta$ is a Levi subalgebra of $\mf g$ in a parabolic decomposition.
  Let $\la\in \h^\ast$ be dominant under the dot-action of elements in $W$. Denote by $M(\la)$ the Verma module over $\mf g_\oa$ of highest weight $\la$ with respect to the triangular decomposition \eqref{eqtraigoa}. We set  $\mc B_\la$ to be the full subcategory of Harish-Chandra $(\mf g,\mf g_\oa)$-bimodules consisting of objects that are annihilated by some power of the annihilator of $M(\la)$.  For a given dominant weight $\mu \in \h^\ast$,  the stabilizer of $\mu$ under the dot-action of $W$ is denoted by $W_\mu$. Also, we set $W_\zeta \subseteq W$ to be the Weyl group of $\mf l_\zeta$, for any $\zeta\in \mc I.$  

There is a version of the category $\mc N(\la, \zeta)$ adapted to our situation of classical Lie superalgebras, which we denote by $\widetilde{\mc N}(\la, \zeta)$; see Section \ref{sect412}. The following is our second main result: 
\begin{thmB} Let $\zeta\in \mc I$ be such that $W_\la= W_\zeta$. Then $\widetilde{\mc N}(\Lambda, \zeta)$ and $\mc B_\la$ are equivalent.
\end{thmB}

We remark some consequences: if $\eta\in \mc I$  satisfies that  $W_\la =   W_\zeta =W_\eta$ then $\widetilde{\mc N}(\la +\Upsilon, \zeta)\cong \widetilde{\mc N}(\la +\Upsilon, \eta)$.  Also, if $s\in W$ is a simple reflection  such that $s$ does not lie in the integral Weyl group of $\la$, then we have $ \widetilde{\mc N}(\la+\Upsilon,\zeta) \cong \widetilde{\mc N}(s\cdot \la +\Upsilon,\eta),$ for any $\eta\in \mc I$ with $W_\eta =W_{s\cdot \la}.$

\subsection{}  We write $M(\la,\zeta)$ instead of $\widetilde{M}(\la, \zeta)$ in the case when $\g=\g_\oa$. As has been mentioned, if $\mf g$ is a classical Lie superalgebra of type I, then the Kac functor gives a bijection of simple Whittaker modules in $\mc N$ and $\widetilde{\mc N}$, leading to the following alternative definition of standard Whittaker module
\begin{align}
&\widetilde{M}(\la, \zeta):= K(M(\la, \zeta)),   
\end{align} for any $\la \in \h^\ast$ and $\zeta\in \mc I.$ This definition can be viewed as a specific case of that given in Theorem A (see Definition \eqref{eq33}) in the case when  $\mf g$ is one of $\gl(m|n),~\mf{osp}(2|2n)$ and $\pn.$ 

Recall that the category $\mc O$ for classical Lie superalgebras admit structures of  highest weight category (see, e.g., \cite{BS,CCC}) with Verma modules $\widetilde{M}(\la)$ indexed by $\la \in \h^\ast$ as standard objects. The following is our  third main result: 
\begin{thmC} \label{mainthm::3rd} 
	Suppose that $\mf g$ is a classical Lie superalgebra of type I.  Then for any $\la,\mu \in \h^\ast$ and $\zeta\in \mc I$, we have 
	\begin{align}
	&[\widetilde{M}(\la, \zeta): \widetilde{L}(\mu, \zeta)] = \sum_{\nu}[\widetilde{M}(\la): \widetilde{L}(\nu)],
	\end{align} where the summation runs over all $\mf n_\zeta$-antidominant weights $\nu$ such that $\mu\in W_\zeta\cdot \nu$. 
\end{thmC}

As a consequence, the composition factors of standard Whittaker modules over the general linear Lie  superalgebras $\mf{gl}(m|n)$  and the ortho-symplectic Lie superalgebras $\mf{osp}(2|2n)$ can be computed by   recent works on the irreducible characters of the BGG category $\mc O$; see, e.g., \cite{Br1, CLW1, CLW2, CMW, Bo1, BaoWang}. 

%Also, if $\mf g$ is a basic classical Lie superalgebra of type I,   using Theorem C we recover the application of  \cite[Corollary 6.8]{CM} on  the criteria for simplicity of the induced module $\widetilde M(\la, \zeta)$. 

\subsection{}
The  paper is organized as follows. In Section \ref{Sect2}, we provide some background materials on
classical Lie superalgebras and Whittaker modules. %In particular, we identify simple objects of $\widetilde{\mc N}$ and that of the Whittaker categories considered in \cite{BCW}. 
In Section \ref{Sect3}, we obtain a classification of standard and simple Whittaker modules for classical Lie superalgebras in terms of their parabolic decompositions. The proof of Theorem A will be given in Section \ref{Sect31}. For classical Lie superalgebras of type I, an alternative definition of standard Whittaker modules will be introduced in Section \ref{sect32} that are to be used in the sequel. In this case, we will classify simple objects of $\widetilde{\mc N}$ in full generality. 

 %, leading to a classification of simple objects of the category $\mc W(\g,\mf n)$ from \cite{BCW} too.
 
In Section \ref{Sect4}, we review Harish-Chandra bimodules, cokernel categories and  Mili{\v{c}}i{\'c}-Soergel equivalence, including several essential ingredients for our main results. %We will establish in Section \ref{Sect43} the equivalence of a subcategory of Whittaker modules and a category of Harish-Chandra $(\mf g, \mf g_\oa)$-bimodules. In particular, the proof of Theorem B follows from Theorem \ref{thm::2}.
Applying tools in \cite{CC} and generalizing \cite{MS}, we will establish in Section \ref{Sect43} the equivalence stated in  Theorem B.

 In Section \ref{Sect5}, we focus on the multiplicity problem of standard Whittaker modules for classical Lie superalgebras of type I. Section \ref{Sect51} is devoted to the proof of Theorem C.  One can also find a detailed example of the general linear Lie superalgebra $\gl(1|2)$ in Section \ref{Sect522ex}. Several various criteria of simplicity for Whittaker modules over Lie superalgebras $\gl(m|n), \mf{osp}(2|2n)$ and $\pn$ are given in  Sections \ref{Sect51}, \ref{Sect52} and \ref{Sect53}. Also, we obtain composition factors of typical standard Whittaker modules  $\widetilde{M}(\la,\zeta)$ over $\pn$  in Section \ref{Sect53}.  For $\g=\gl(m|n), \mf{osp}(2|2n)$, we put together all the results from previous sections and reduce the problem of composition factors in certain standard Harish-Chandra bimodules to Kazhdan-Lusztig combinatorics in Section \ref{sect55}. 
\vskip 0.1cm  
{\bf Acknowledgment}. The author was supported by a MoST grant, and he would like to thank  Shun-Jen Cheng, Kevin Coulembier,  Volodymyr Mazorchuk and Weiqiang Wang for interesting discussions and helpful comments.

\section{Preliminaries} \label{Sect2}
\subsection{}
Let $\mf g$ be a finite-dimensional complex classical Lie superalgebra.  Denote by   $U(\mf g)$  the         universal enveloping algebra of  $\mf g$.  Let  $Z(\mf g)$ be the center of  $U(\mf g)$.  For a given $\la \in \h^\ast$, we set $\chi_{\la}^{\mf g}:  Z(\mf g)\rightarrow \mathbb C$ to be the central character associated with $\la$. We sometimes  write $\chi_\la^{\oa}$ instead of $\chi_\la^{\mf g_\oa}$. Recall that we fixed a triangular  decomposition  $\mf g =\mf n^- \oplus \mf h \oplus \mf n$  with even Cartan subalgebra $\mf h=\mf h_\oa$.

The Weyl group $W$ of $\mf g$ is by definition the Weyl group $W$ of $\mf g_\oa$ with  its defining action on $\mf h^\ast$. The usual dot-action of $W$ on  $\mf h^\ast$ is defined as $w\cdot \la :=w(\la+\rho_\oa)-\rho_\oa,$ for any  $ w\in W$ and $~\la \in \h^\ast$, 
where $\rho_\oa$ is the half of the sum of all positive roots of $\mf g_\oa$. A weight is called {\em integral, dominant or anti-dominant} if it is integral,
dominant or anti-dominant as a $\mf g_\oa$-weight, respectively. We recall that $W_\la$ denotes the stabilizer of $\la$ under the dot-action of $W$, for any weight $\la \in \h^\ast$.  For any $\zeta\in \mc I$, we  denote by $W_\zeta\subseteq W$ the Weyl group of $\mf l_\zeta$, which naturally acts on $\h^\ast$ via consistent dot-action. For any weight module $M$, we set $P(M)$ to be the set of all weights in $M$.  Finally, we denote by $\Upsilon\subset \h^\ast$ the set of integral weights. Then  we have $\mZ\Phi \subset\Upsilon$. 

 %A weight $\lambda$ is dominant if there exists no $w\in W$ such  that~$w\cdot\lambda-\lambda$ is a non-empty sum of elements in~$\Phi_{\oa}^+$.

 %We denote the $\rho_{\oa}$-shifted action of the Weyl group $W=W(\fg_{\oa}:\fh_{\oa})$ on $\fh_{\oa}^\ast$ by~$w\cdot\lambda=w(\lambda+\rho_{\oa})-\rho_{\oa}$.

%We define $\Gamma=\mZ\Phi$. 

\subsection{} Let $\mf k$ be a Lie superalgebra. For a subalgebra $\mf s\subset\mf k$, we denote by $\Res^{\mf k}_{\mf s}$ the restriction functor from $\mf k$ to $\mf s$. We have exact induction and coinduction functors
$$\Ind^{\mf k}_{\mf s}(-)=U(\mf k)\otimes_{U(\mf s)}-\qquad\mbox{and}\qquad \Coind^{\mf k}_{\mf s}(-)=\Hom_{U(\mf s)}(U(\mf k),-).$$
They are left and right adjoint functors to $\Res^{\mf k}_{\mf s}$. 
If $\mf s$ contains $\mf k_{\oa}$, then \cite[Theorem~2.2]{BF} (see also \cite{Go}) implies that
$\Ind^{\mf k}_{\mf s}(-)\;\cong\; \Coind^{\mf k}_{\mf s}(\Lambda^{\text{top}}(\mf k/\mf s)\otimes -).$
%In other words, the induction and coinduction functors are isomorphic up to taking the tensor product with the one-dimensional $\fc$-module realised as the top symmetric power of the purely odd superspace $\fa/\fc$.

%We will use the undecorated notations $\Res$, $\Ind$, $\Coind$ to  refer to these functors acting when  $\mf k=\mf g$ and $\mf s = \mf g_\oa$.
We will use the undecorated notations \[\text{Ind, Coind}: \mf g_\oa\text{-Mod}\rightarrow \mf g\text{-Mod}, ~\Res: \mf g\text{-Mod} \rightarrow \mf g_\oa\text{-Mod},\] to  refer to these  functors  when  $\mf k=\mf g$ and $\mf s = \mf g_\oa$.

%We have exact induction, coinduction and restriction functors  \[\text{Ind, Coind}: \mf g_\oa\text{-Mod}\rightarrow \mf g\text{-Mod}, ~\Res: \mf g\text{-Mod} \rightarrow \mf g_\oa\text{-Mod}.\]

In particular, the functors $\mathrm{Ind}$ and $\mathrm{Coind}$ are isomorphic, up to the equivalence given by tensoring with the one-dimensional $\fg_\oa$-module $\Lambda^{\text{top}} \mf g_\ob$ on the top degree subspace of $\Lambda \g_{{\bar{1}}}$.

\begin{prop} \label{lem::simples}
	Suppose that $S$ is a simple $\g$-module. Then $\Res S$ is locally finite over $Z(\mf g_\oa)$. 
\end{prop}
\begin{proof} We adapt the argument   in \cite[Lemma 4.2]{CM} to complete the proof. Since $U(\mf g)$ is a  finitely-generated $U(\mf g_\oa)$-module,  the  $\g_\oa$-module $\Res S$ is finitely-generated. Therefore  $\Res S$ has a simple quotient $V$. By adjunction we have 
	\[\text{Hom}_{\mf g}(S, \Coind V) = \text{Hom}_{\mf g_\oa}(\Res S, V)\neq 0,\] which implies that $S$ is a submodule of $\Ind W,$ where $W:=\Lambda^{\text{top}} \mf g_\ob^\ast\otimes V$ is a simple $\mf g_\oa$-module.  We note that the module $\Res \Ind W\cong U(\mf g_\ob) \otimes  W$ is locally finite over $  Z(\mf g_\oa)$ by \cite[Section 2.6]{BG}. This completes the proof.
\end{proof}

\begin{rem} \label{Rmk2}To compare simple objects of the category $\mc W(\g,\mf n)$ from \cite{BCW} to that of our category  $\widetilde{\mc N}$, we note that    any simple $\mf g$-module that is locally finite over $\mf n$ lies in $\widetilde{\mc N}$ by Lemma \ref{lem::simples}. Namely, we remark that $\mc W(\g, \mf n)$ and $\widetilde{\mc N}$ have the same collection of simple objects. %Namely, we remark that a simple $\mf g$-module $S$ is locally finite over $\mf n_\oa$ if and only if  $S\in \widetilde{\mc N}$. 
\end{rem}

%Let $\mc N$ denote the full subcategory of $\mf g_\oa\mod$ consisting of those finitely generated $\mf g_\oa$-modules that are locally finite over $\mf n_\oa$ and $\mc Z(\mf g_\oa)$. Let $\widetilde{\mc N}$ denote the full subcategory of $\mf g\mod$ consisting of those finitely generated $\mf g$-modules that are locally finite over $\mf n$ and $\mc Z(\mf g_\oa)$. 

\begin{lem} \label{lem::3}
The functors $\emph{Ind, Coind}$ and $\Res$ restrict to exact functors between $\mc N$ and $\widetilde{\mc N}$:
\[\emph{Ind, Coind}: \mc N\rightarrow \widetilde{\mc N}, ~\Res: \widetilde{\mc N} \rightarrow \mc N.\]
\end{lem}
\begin{proof} 
	Let $M\in \widetilde{\mc N}$. Then $M$ is a finite sum of cyclic $U(\g)$-submodules. We note that $U(\mf g)$ is finitely-generated over $U(\mf g_\oa)$ since $U(\mf g_\ob)$ is finite-dimensional. This means that  $\Res M\in \mc N$. 
	
	Conversely, assume that $V\in \mc N$. Then  $\Res \Ind V \cong U(\g_\ob)\otimes V$ is locally finite over $Z(\mf g_\oa)$ (see, e.g., \cite[Section 2.6]{BG}). Also, $U(\g_\ob)\otimes V$ is locally finite over $\n_\oa$, and so $\Ind V \in \widetilde{\mc N}.$
\end{proof}

As a consequence of Lemma \ref{lem::3}, we have the following corollary.

\begin{cor} \label{prop::3}
	Every object in $\widetilde{\mc N}$ has finite length. 
\end{cor}
\begin{proof}
	By \cite[Theorem 2.6]{MS}, every object in $\mc N$ has finite length (see also \cite{Mc} and \cite[Corolloary 4]{CoM}). The conclusion follows.
\end{proof}

For any $\zeta \in \mc I,$ we define the following blocks (see also \cite[Section 1]{MS}): \begin{align}
&\mc N(\zeta):=\{N \in \mc N |~x-\zeta(x) \text{ acts locally nilpotently on }N, \text{for any } x\in \mf n_\oa\},\\
&\widetilde{\mc N}(\zeta):=\{M \in \tN |~x-\zeta(x) \text{ acts locally nilpotently on }M, \text{for any } x\in \mf n_\oa\}.
\end{align}
 
We may observe that $\Ind, \Coind$ and $\Res$ restrict to well-defined functors between  $\mc N(\zeta)$ and $\tN(\zeta)$.  

 % There have been some   general approaches to the problem of extension fullness of categories developed in \cite{CoM}. Using \cite[Proposition 1]{CoM} and Lemma \ref{lem::3} we have the following corollary.  
%\begin{cor}	The category $\widetilde{\mc N}$ is extension full in $\mf g$-{\em Mod} in the sense of \cite[Section 2.2]{CoM}. Namely, the natural inclusion functor $i: \widetilde{\mc N}\rightarrow \mf g\emph{-Mod}$ induces the following isomorphisms of extension groups \begin{align*}	&i^d_{K,M}: ~\Ext_{\widetilde{\mc N}}^d(K, M) \cong \Ext_{\mf g\emph{-Mod}}^d(K, M),	\end{align*} for every $K, M\in \widetilde{\mc N}$ and  $d\geq 0$. \end{cor} \begin{proof} By Lemma \ref{lem::3}, each $M \in \widetilde{\mc N}$ is a quotient of $\Ind \Res M$, where $\Res M\in \mc N$. We apply \cite[Proposition 1]{CoM} for $\mc A=\mf g$-Mod, $\mc B=\widetilde{\mc N}$ and $\mc B_0$ being the full subcategory of $\mc B$ consisting of all modules isomorphic to $\Ind V$, for some $V\in \mc N$. For any $V\in \mc N$, $M \in \widetilde{\mc N}$ and   $d\geq 0$, we have 	\begin{align*} &\Ext_{\widetilde{\mc N}}^d(\Ind V, M) 	\cong  \Ext_{\mc N}^d(V, \Res M)  \cong \Ext_{\mf g_\oa\text{-Mod}}^d(V, \Res M) \cong 	\Ext_{{\mf g\text{-Mod}}}^d(\Ind V, M) 	\end{align*} The conclusion follows from  \cite[Proposition 1]{CoM}.\end{proof}

\section{Simple and standard Whittaker modules} \label{Sect3}
In this section, we define the various generalizations of  standard Whittaker modules and study their fundamental  properties in our setup. We will classify the simple Whittaker modules in terms of parabolic decompositions for an arbitrary classical Lie superalgebra. For Lie superalgebras of type I, we provide a complete classification of simple Whittaker modules using the Kac functors. 

\subsection{Simple Whittaker modules: arbitrary classical  Lie superalgebras}   \label{Sect31} Let $\g$ be an arbitrary  finite-dimensional complex classical Lie superalgebra. 
For each $\zeta \in \mc I$,  %we denote by $\mf p_\zeta$ the corresponding parabolic subalgebra, namely, $\mf p_\zeta:=\mf l_\zeta \oplus \mf u_\zeta$. Also, 
we denote by $\mf l_\zeta = \mf n_\zeta^-\oplus \mf h\oplus \mf n_\zeta$ the corresponding triangular decomposition of $\mf l_\zeta$. For any $\la \in \mf h^\ast$, we recall  that Kostant's simple Whittaker modules are defined as follows %(see  \cite{Ko78})
\begin{align}
&Y_\zeta(\la, \zeta):=U(\mf l_\zeta)/\text{Ker}(\chi^{\mf l_\zeta}_\la) U(\mf l_\zeta) \otimes_{U(\mf n_\zeta)}\mathbb C_\zeta,
\end{align} where  $\text{Ker}(\chi^{\mf l_\zeta}_\la)$ is the kernel of the central character $\chi_\la^{\mf l_\zeta}$ of $\mf l_\zeta$ and $\C_\zeta$ is the one-dimensional $\mf n_\zeta$-module associated with $\zeta$. The isomorphism $Y_\zeta(\la, \zeta)\cong Y_\zeta(\mu,\zeta)$  holds if and only if $W_\zeta\cdot \la =W_\zeta\cdot \mu.$

Suppose that $\mf l_\zeta$ is a Levi subalgebra in a parabolic decomposition $\mf g=\mf u_\zeta^- \oplus \mf l_\zeta\oplus \mf u_\zeta$ of $\g$. The {\em standard Whittaker modules} over $\g_\oa$ and $\g$ are  respectively  defined as
\begin{align} 
&M(\la, \zeta):= U(\g_\oa)\otimes_{\mf p_\oa} Y_\zeta(\la, \zeta), \label{eq32}\\
&\widetilde{M}(\la, \zeta):= {U}(\g)\otimes_{\mf p} Y_\zeta(\la, \zeta),\label{eq33}
\end{align} where $\mf p:=\mf p_\zeta=\mf l_\zeta \oplus \mf u_\zeta$ is the corresponding parabolic subalgebra. The module $M(\la,\zeta)\in \mc N$ has been studied in \cite{Mc,MS} (see also \cite{B}).

The following result is proved in \cite[Proposition 2.1]{MS}.
\begin{lem}[Mili{\v{c}}i{\'c}-Soergel] \label{lem::MS21}
	For any $\la \in \mf h^\ast$ and $\zeta\in \mc I$, the standard Whittaker $\mf g_\oa$-module $M(\la, \zeta)$ has simple top $L(\la,\zeta)$. Let $\mu \in \h^\ast$, then $${L}(\la, \zeta)\cong  {L}(\mu, \zeta)\Leftrightarrow {M}(\la, \zeta)\cong {M}(\mu, \zeta)\Leftrightarrow W_\zeta \cdot \la =W_\zeta \cdot \mu.$$ Every simple module in $\mc N(\zeta)$ is  of the form $L(\la, \zeta),$ for $\la \in \h^\ast.$
\end{lem}

%Let $\mf g$ be a classical Lie superalgebra. Recall the definition of parabolic decomposition from \ref{deflu}.

We are now in a position to prove our first main result. 
The conclusion of Theorem B is a consequence of the following theorem.
\begin{thm} \label{mainthm1}
	Let $\zeta\in  \mc I$. Suppose that $\mf l_\zeta$ is a Levi subalgebra in a parabolic decomposition $\mf g=  {\mf u}^-_\zeta \oplus \mf l_\zeta \oplus {\mf u_\zeta}$. Then we have 
	\begin{itemize}
		\item[(1)] $\widetilde{M}(\la,\zeta)$ has simple top, which is denoted by $\widetilde{L}(\la, \zeta)$, for each $\la \in \mf h^\ast$.
		\item[(2)] $\{\widetilde{L}(\la, \zeta)|~\la \in \h^\ast\}$ is the complete list of simple modules in $\widetilde{\mc N}(\zeta)$. 
		\item[(3)] For any $\la, \mu \in \h^\ast$, the  following are equivalent:
		\begin{itemize}
			\item[(a)] $\widetilde{M}(\la, \zeta)\cong \widetilde{M}(\mu, \zeta)$. 
			\item[(b)] $\widetilde{L}(\la, \zeta)\cong \widetilde{L}(\mu, \zeta)$.
			\item[(c)] $W_\zeta \cdot \la =W_\zeta\cdot \mu$.
		\end{itemize}  
	\end{itemize}
\end{thm}
\begin{proof} We first claim that $\widetilde{M}(\la, \zeta)\in \widetilde{\mc N}$. To see this, we may observe that   $\Res \widetilde{M}(\la, \zeta)$ is an epimorphic image of the $\g_\oa$-module $U(\g_\ob)\otimes M(\la, \zeta)$ by the Poincar\'e-Birkhoff-Witt basis theorem. Therefore $\widetilde{M}(\la, \zeta)$ is locally finite over $Z(\g_\oa)$ (see, e.g., \cite[Section 2.6]{BG}) and over $\mf n$. We may conclude that $\widetilde{M}(\la, \zeta)\in \widetilde{\mc N}$ since it is generated by any non-zero vector of $Y_\zeta(\la,\zeta)\subset \Res \widetilde{M}(\la,\zeta)$.  
	
	Next,  we shall proceed with an argument similar to the proof of \cite[Proposition 2.1]{MS}. % Decompose $\mf h =\mf h_\zeta\oplus \mf h^\zeta$, where $\mf h_\zeta:=\mf h\cap [\mf l_\zeta,\mf l_\zeta]$ and  $\mf h^\zeta:=\bigcap_{\alpha \in \Phi_\zeta} \text{ker}\alpha$. 
  We may note that  $H\in   \bigcap_{\alpha \in \Phi_\zeta} \text{Ker}(\alpha)$ and so $H$ acts on $Y_\zeta(\la,\zeta)$ via $\la(H)$.  Therefore $\widetilde{M}(\la ,\zeta)$ decomposes into eigenspaces $\widetilde{M}(\la,\zeta)_k$ with $k \in \la(H) +\sum_{\alpha\in P(\mf u^-_\zeta)} \mathbb Z_{\geq 0}\alpha(H)$ according to the eigenvalues of the action $H$. Since $\alpha(H)<0$ for any $\alpha \in P(\mf u^-_\zeta),$ it follows that $\widetilde{M}(\la,\zeta)_{\la(H)} = Y_\zeta(\la, \zeta)$. Also, all the $\widetilde{M}(\la,\zeta)_k$ are $\mf l_\zeta$-submodules since $\alpha(H)=0$ for any $\alpha \in \Phi_\zeta$. 
  
  Let $N$ be a proper submodule of $\widetilde{M}(\la, \zeta)$, then $N$ decomposes $N =\bigoplus_k N_k$ with $\mf l_\zeta$-submodules $N_k\subseteq \widetilde{M}(\la,\zeta)_k$ according to the eigenvalues $k$ of $H$ acting on $N$. Since $Y_\zeta(\la, \zeta)$
is simple, we may conclude that $N_{\la(H)} =0$. Therefore  $\widetilde{M}(\la, \zeta)$ has a unique maximal submodule. This proves Part (1). We denote the simple top of $\widetilde{M}(\la, \zeta)$ by $\widetilde{L}(\la, \zeta).$
   %For any $\la \in \mf h^\ast$ we denote by $\la^f$ its restriction to $\mf h^f$.  Note that $H\in \mf h^f$ be definition, which means that $\widetilde{M}(\la ,f)$ decomposes under into eigenspaces $\widetilde{M}(\la,f)_k$ according eigenvalues of $\text{ad}H$ with $k \in \la^f(H) -\sum_{\alpha\in \Pi} \mathbb Z_{\geq 0}\alpha^f(H)$.

  Next we prove Part $(2)$. Recall that we put $\mf p:=\mf p_\zeta$. Let $S\in \widetilde{\mc N}(\zeta)$ be a simple module. Since $S$ has finite length, there exists $\mu \in \mf h^\ast$ such that $L(\mu, \zeta)\hookrightarrow \Res S$. Therefore we have  
  \begin{align}
  &\text{Hom}_{\mf g}(\Ind M(\mu ,\zeta), S) = \text{Hom}_{\mf g_\oa}(M(\mu, \zeta), \Res S) \neq 0.
  \end{align} 
  The $\mf l_\zeta$-module $\Res_{\mf l_\zeta}^{\mf p}\Ind^{\mf p}_{\mf p_\oa} Y_\zeta(\mu, \zeta)\cong U(\mf p_\ob)\otimes Y_\zeta(\la,\zeta)$ have composition factors of the form $Y_\zeta(\la+\nu,\zeta),$ for $\nu\in P(U(\mf p_\ob))$ (see, e.g.,  \cite[Theorem 4.6]{Ko78}). Note that $\nu(H)>0$, for any $\nu\in P(U(\mf p_\ob))$. Therefore the $\mf p$-module  $\Ind^{\mf p}_{\mf p_\oa} Y_\zeta(\mu, \zeta)$ has a filtration with simple subquotients $Y_{\zeta}(\la+\nu,\zeta)$'s, which are annihilated by $\mf u_\zeta$. We may observe that 
  \begin{align}
  &\Ind M(\mu ,\zeta) = \Ind_{\mf g_\oa}^\g\Ind_{\mf p_\oa}^{\mf g_\oa} Y_\zeta(\mu, \zeta)\cong \Ind_{\mf p}^\g\Ind^{\mf p}_{\mf p_\oa} Y_\zeta(\mu, \zeta),
  \end{align} which implies that $\Ind M(\mu, \zeta)$ admits a filtration of standard Whittaker modules. Consequently, $S$ is a composition factor of $\widetilde{M}(\la, \zeta)$ for some $\la \in \h^\ast$. It remains to show that every composition factor of $\widetilde{M}(\la, \zeta)$ is of the form $\widetilde{L}(\la', \zeta)$, for $\la'\in \h^\ast$.

   Let $S$ be a composition factor of $\widetilde{M}(\la,\zeta)$. Again, under the action of $H$ the $S$ decomposes  into eigenspaces $S_k$ with $k \in \la(H) +\sum_{\alpha\in P(\mf u^-_\zeta)}\mathbb Z_{\geq 0}\alpha(H)$. Since $\mf u_\zeta S_k \subset S_{k'}$ with $k<k'$, we may conclude that there exists an eigenvalue $m$ such that $S_m\neq 0$ and $\mf u_\zeta S_m=0$. Since $S_m$ is a $\mf l_\zeta$-submodule of $\widetilde{M}(\la, \zeta)_m$, it follows from \cite[Theorem 4.6]{Ko78} that $S_m$ has a simple submodule $Y_\zeta(\gamma,\zeta)$, for some $\gamma \in \mf h^\ast$. Consequently, we have 
   \begin{align}
  &\text{Hom}_{\mf g}(\widetilde{M}(\gamma, \zeta), S) = \text{Hom}_{\mf p_\zeta}(Y_\zeta(\gamma ,\zeta), \Res^{\mf g}_{\mf p_\zeta}S)  \neq 0.
  \end{align} Therefore $S\cong \widetilde{L}(\gamma, \zeta)$. This proves Part (2).
  
  We have known $(c)\Rightarrow (a)$ already. We now prove $(a)\Rightarrow (b)$. If $\widetilde{M}(\la, \zeta) \cong \widetilde{M}(\mu, \zeta)$ then $\widetilde{M}(\la, \zeta)_{\la(H)} =Y_\zeta(\la, \zeta)$ and  $\widetilde{M}(\mu, \zeta)_{\mu(H)} =Y_\zeta(\mu,\zeta)$ are isomorphic as $\mf l_\zeta$-modules, and so $W_\zeta\cdot \la =W_\zeta \cdot \mu$ by Lemma \ref{lem::MS21}. 
  
  Finally, we prove the direction $(b)\Rightarrow(c)$. Again, $\widetilde{L}(\la,\zeta)$ decomposes into eigenspaces $\widetilde{L}(\la,\zeta)_k$ with $k \in \la(H) +\sum_{\alpha\in P(\mf u^-_\zeta)} \mathbb Z_{\geq 0}\alpha(H)$ according to the eigenvalues of the action of $H$ on  $\widetilde{L}(\la,\zeta)$. Thus, we have $\widetilde{L}(\la, \zeta)_{\la(H)} =Y_\zeta(\la,\zeta)$ and  $\widetilde{L}(\mu, \zeta)_{\mu(H)} =Y_\zeta(\mu, \zeta)$, which implies  $W_\zeta\cdot \la =W_\zeta\cdot \mu$ by Lemma \ref{lem::MS21}. This completes the proof.
\end{proof}
\subsection{Simple Whittaker modules: Lie superalgebras of type I}\label{sect32}
In this subsection, we let $\g=\mf g_{-1}\oplus \mf g_0\oplus \mf g_1$  be a finite-dimensional complex classical Lie superalgebra of type I. We will redefine the standard Whittaker modules in this case, leading to  a complete classification of simple Whittaker $\g$-modules.  The advantage is that we do not need to assume that  $\mf l_\zeta$ is a Levi subalgebra in a parabolic decomposition of $\mf g$.

For a given $\mf g_\oo$-module $V$, we can extend $V$ trivially to a $\mf g_\oo \oplus \mf g_{1}$-module and define the {\em Kac module} of $V$ as $K(V) :=
\text{Ind}_{\mf g_{\geq 0}}^{\mathfrak{g}}(V).$ Then this defines an exact functor $K(\cdot): \mf g_\oo\text{-Mod}\rightarrow \mf g\text{-Mod}$, which we call {\em Kac functor} (see also \cite[Sections 2.4,  3]{CM}). We  observe that $K(V) \cong \Lambda(\mf g_{-1}) \otimes V$ as vector spaces. Throughout this subsection, for any $\la \in \h^\ast$ and $\zeta\in \mc I$ we define the {\em standard Whittaker module} for type I Lie superalgebra   as
\begin{align}
&\widetilde{M}(\la, \zeta):= K(M(\la, \zeta)), \label{eq::WhiKac} 
\end{align}  where $M(\la, \zeta)$ is the standard Whittaker module over $\mf g_\oa$ as  in \eqref{eq32}. We remark that the definition  \ref{eq::WhiKac} can be viewed as special cases of \eqref{eq33}  when  $\mf g$ is one of $\gl(m|n),~\mf{osp}(2|2n)$ and $\pn.$

\begin{lem}
	The Kac functor $K(-)$ defines an exact functor from $\mc N$ to $\widetilde{\mc N}$.
	\end{lem}
\begin{proof} Let $M\in \mc N$. We note that  $\Res K(M) \cong \Lambda(\mf g_{-1})\otimes M$ as $\mf g_\oa$-modules. Therefore $K(M)$ is locally finite over $\mf n$. It follows from \cite[Section 2.3, 2.6]{BG} that $K(M)$ is a finitely-generated $\mf g$-module and is locally finite over $Z(\g_\oa)$, proving the claim.
\end{proof}

The following theorem is an analog of Theorem \ref{mainthm1}, but we do not need to assume that $\mf l_\zeta$ is a Levi subalgebra of $\mf g$.
\begin{thm} \label{mainthm1typeI} 
	 Let $\zeta \in \mc I$. Then we have 
		\begin{itemize}
			\item[(1)] $\widetilde{M}(\la,\zeta)$ has simple top, which is denoted by $\widetilde{L}(\la, \zeta)$, for each $\la \in \mf h^\ast$.
			\item[(2)] $\{\widetilde{L}(\la, \zeta)|~\la \in \h^\ast\}$ is the complete list of simple modules in $\widetilde{\mc N}(\zeta)$. 
			\item[(3)] For any $\la,\mu \in \h^\ast$, the following are equivalent:
			\begin{itemize}
				\item[(a)] $\widetilde{M}(\la, \zeta)\cong \widetilde{M}(\mu, \zeta)$. 
				\item[(b)] $\widetilde{L}(\la, \zeta)\cong \widetilde{L}(\mu, \zeta)$.
				\item[(c)] $W_\zeta \cdot \la =W_\zeta\cdot \mu$.
			\end{itemize}  
		\end{itemize}
\end{thm}
\begin{proof}
	 %Next we assume that $f$ is regular, which implies that we have the triangular decomposition $\mf g=\mf u^-_\ob\oplus \mf g_\oa \oplus \mf u_\ob$ for this case. Then $[\mf u_\ob,\mf u_\ob] = [\mf u^-_\ob,\mf u^-_\ob]=0$ since $\mf u_\ob$ is an ideal. This means that $\mf g$ is a classical Lie superalgebra of type I. We now set $\mf g_1:= \mf u$ and $\mf g_{-1}:=\mf u^-$. We observe that $\widetilde{M}(\la, f) = \Ind^\g_{\mf g_\oa +\mf g_1} M(\la, f)$ for this case. 
	 For any $M\in \g$-Mod, we define $M^{\g_{\pm 1}}$ to be the subspaces of $\g_{\pm 1}$-invariants, namely, $M^{\g_{\pm 1}}:=\{m\in M|~\g_{\pm 1}m=0\}.$ 
	We shall first adapt the argument in \cite[Lemma 4.4]{CM} to complete the proof of Part $(1)$. 
	Suppose that $\phi: \widetilde{M}(\la, \zeta) \rightarrow S$ is a simple quotient of $\widetilde{M}(\la,\zeta)$. By  \cite[Theorem 4.1 (ii)]{CM} there is a simple $\mf g_\oa$-module $V$ such that 
	\begin{align}
	&S \hookrightarrow  \Ind^\g_{\mf g_\oa+\mf g_{-1}} V. 
	\end{align}
  By definition the $\mf g_\oa$-module $M(\la, \zeta)$ can be regarded as a $\mf g_\oa$-submodule of $\Res\widetilde{M}(\la,\zeta)$. Note that $\phi(M(\la, \zeta))\subseteq \Lambda^{ \text{top}}\mf g_1 \otimes V$ since $\mf g_1 \cdot M(\la, \zeta)=0$ and the subspace of $\mf g_1$-invariants of $\Ind^\g_{\mf g_\oa+\mf g_{-1}} V$ is $\Lambda^{ \text{top}}\mf g_1 \otimes V.$ But $\Lambda^{ \text{top}}\mf g_1 \otimes V$ is a simple $\mf g_\oa$-module, we may conclude that $\Lambda^{ \text{top}}\mf g_1 \otimes V\cong L(\la, \zeta)$ by Lemma \ref{lem::MS21}. Therefore we obtain that $V\cong \Lambda^{ \text{top}}\mf g_1^\ast \otimes L(\la, \zeta)$. 
  
  Finally, using adjunction and Schur's lemma, it follows that 
  $$\dim \text{Hom}_{\g}(\widetilde{M}(\la, \zeta), \Ind^\g_{\mf g_0+\mf g_{-1}} V) = \dim \text{Hom}_{\g_{\geq 0}}(M(\la, \zeta), \Lambda^{\text{top}}\mf g_{1}\otimes V)=1,$$
  by Lemma \ref{lem::MS21}. This proves Part $(1)$.
  
  We are going to prove $(2)$. Let $S\in \widetilde{\mc N}(\zeta)$ be a simple module. We claim that the subspace $S^{\g_1}$ of  $\g_1$-invariant elements is a simple $\mf g_\oa$-module in $\mc N(\zeta)$. It suffices to show that $S^{\g_1}$ is simple. To see this, we note that $S$ is isomorphic to the socle of $\Ind_{\g_0+\g_{-1}}^\g V$, for some simple $\g_\oa$-module $V$ by \cite[Corollary 4.3]{CM}. By \cite[Lemma 3.2]{CM}, the simple $\g_\oa$-submodule  $(\Ind_{\g_0+\g_{-1}}^\g V)^{\g_{1}} =\Lambda^{\text{top}}\g_1\otimes V$  generates the simple socle $S \cong \text{soc}(\Ind_{\g_0+\g_{-1}}^\g V)$ of $\Ind_{\g_0+\g_{-1}}^\g V$ as a $\mf g$-submodule. It follows that $S^{\g_1}$ is simple.   By Lemma \ref{lem::MS21}, there exists $\la'\in \mf h^\ast$ such that $S^{\g_1}\cong L(\la',\zeta)$. It follows from \cite[Theorem 4.1]{CM} that $S$ is the quotient of %$ \Ind_{\mf g_\oa+\g_1}^{\g}(M(\la',\zeta))\cong 
  $ \widetilde{M}(\la', \zeta)$, that is, $S\cong \widetilde{L}(\la',\zeta)$ by definition. 
	
	 Finally, it remains to show Part $(3)$. For any $\la \in \h^\ast$, since $\widetilde{L}(\la, \zeta)$ is the simple top of $K(L(\la, \zeta))$, we know that  $\widetilde{L}(\la, \zeta)\cong\widetilde{L}(\mu, \zeta)$ if and only if ${L}(\la, \zeta)\cong {L}(\mu, \zeta)$ by  \cite[Theorem 4.1]{CM}.  Thus, Parts $(b),(c)$ are equivalent by Lemma \ref{lem::MS21}. It remains to show   the direction $(a)\Rightarrow (b)$.

	  Suppose $(a)$ holds. Then we have the following isomorphisms of $\mf g_\oa $-modules
	\[ \Lambda^\text{top}\mf g_{-1} \otimes {M}(\la, \zeta)= \widetilde{M}(\la, \zeta)^{\g_{-1}}\cong  \widetilde{M}(\mu, \zeta)^{\g_{-1}} =  \Lambda^\text{top}\mf g_{-1} \otimes  {M}(\mu, \zeta) .\] Therefore the conclusion of Part $(3)$ follows from Lemma \ref{lem::MS21}. 
\end{proof}

We should mention that the Part (1) in Theorem \ref{mainthm1typeI} generalizes the construction of simple modules in \cite[Section 4]{BCW}, where the case of some basic classical Lie superalgebras of type I  was   considered.

\section{Whittaker modules and Harish-Chandra bimodules} \label{Sect4}

In this section, we continue to assume that  $\mf g$ is a  classical Lie superalgebra with $\zeta\in \mc I$. %We assume that either $\mf g$ is of type I, or $\mf l_\zeta$ is a Levi subalgebra in a parabolic decomposition $\mf g=\mf u^-_\zeta \oplus \mf l_\zeta\oplus \mf u_\zeta$. %In the later case, we set $\mf p_\zeta:=\mf l_\zeta \oplus \mf u_\zeta$ to be the corresponding parabolic subalgebra. 
%We continue to denote by 
Recall that $\mf l_\zeta = \mf n_\zeta^-\oplus \mf h\oplus \mf n_\zeta$ denotes the corresponding triangular decomposition of $\mf l_\zeta$. For a given weight $\la\in \h^\ast$, we set  $\Lambda:=\la+\Upsilon$.

%with $\mf g =\mf n^-\oplus \mf h \oplus \mf n$ be a triangular decomposition of $\mf g$ such that $\mf h$ is even, i.e. $\mf h=\mf h_\oa$.
\subsection{Action of the ring $\hat S^W$} \label{subsect::32}
 \subsubsection{}
In this subsection, we recall the action of the ring $\hat S^W$ from  \cite[Section 4, Section 5]{MS}: let $\hat S$ denote the completion of the symmetry algebra over $\mf h$ over the maximal ideal generated by $\mf h$. Denote by $\hat S^W$ its invariants under the action of the Weyl group $W_\zeta$. We will give a   natural action of $\hat S^W$ on modules in $\widetilde{\mc N}$ using results in \cite[Sections 4, 5]{MS}.

  We note that modules in $\widetilde{\mc N}(\zeta)$ restrict to $\mf l_\zeta$-modules that are locally finite over $Z(\mf l_\zeta)$. Therefore for any    $M\in \widetilde{\mc N}(\zeta)$ there is a ring homomorphism  
\begin{align} 
&\theta_M: ~\hat S^W\rightarrow \text{End}_{\mf l_\zeta}M, \label{eq43}
\end{align}
as constructed in \cite[Section 4]{MS}. As has been proved in \cite[Theorem 4.1]{MS} (see also \cite[Section 5]{MS}), the action of elements of $\hat S^W$ via $\theta_M$  in \eqref{eq43} commute with the $\mf g_\oa$-action on $M$. As an analog for our setup, we remark the following corollary, but we will not use it. 
\begin{prop} We have  $\theta_M(\hat S^W)\subset {\emph End}_{\mf g}M$, for any $M\in \widetilde{\mc N}$. 
\end{prop}
\begin{proof}
	Let $\text{Id}^{\zeta}$ denote the identity functor on the category of $\mf l_\zeta$-modules that are locally finite over $Z(\mf l_\zeta)$. By the construction   in \cite[Section 4]{MS}, for each $s\in \hat S^W$ the element $\theta_M(s)$ is the evaluation of an endomorphism $\theta(s): \text{Id}^{\zeta}\rightarrow \text{Id}^{\zeta}$ at $M$, where $\theta$ is a ring homomorphism from $\hat S^W$ to the endomorphism ring of functor $\text{Id}^\zeta$. 	We note that the adjoint representation $\ad \mf g$  is semisimple over $\mf l_\zeta$. Let $\pi_M: \ad \g \otimes M\rightarrow M$ be the canonical epimorphism. 
 By \cite[Theorem 4.1]{MS} we have   %$$\text{Id}^{\mf l_f}_{\ad\g\otimes M} = \text{Id}^{\mf l_f}_{\ad(\mf g)}\otimes \theta_M(s),$$ 
$$\theta_{\ad\mf g\otimes M}(s)= \text{Id}^{\zeta}_{\ad\g}\otimes \theta_M(s),$$ which gives rise to the identity
$\pi_M\circ(\text{Id}^{\zeta}_{\ad\g}\otimes \theta_M(s)) = \theta_{M}(s)\circ \pi_M$, for any $s\in \hat S^{W}.$  This completes the proof.  
\end{proof}

\subsubsection{} \label{sect412}
We   recall some results and notations in \cite[Section 5]{MS} as follows.  For any positive integer $n$, we define $\mc N(\zeta)^n:=\{M\in \mc N(\zeta)|~\theta(\mf m)^nM=0\},$ where $\mf m$ is the maximal ideal of $\hat S^{W}$. Set $\mf l:=\mf l_\zeta$. Any module $M\in \mc N(\zeta)$ decomposes into generalized eigenspaces $M=\bigoplus_{\la\in \h^\ast}M_{\chi_\la^{\mf l}}$ according to the action of elements of $ Z(\mf l)$.%, where $\chi_\la^{\mf l_f}$ runs over all central character of $\mf l_f$. 

  Recall that we denote the set of integral weights by $\Upsilon$. For any coset $\Lambda= \la +\Upsilon$ in $ \mf h^\ast/\Upsilon$ we put 
\begin{align*}
&\mc N(\Lambda,\zeta) := \{M\in \mc N(\zeta)|~M_{\chi^{\mf l}} = 0 \text{ unless }\chi^{\mf l} = \chi^{\mf l}_\mu, \text{ for some }\mu \in \Lambda\},\\
&\mc N(\Lambda,\zeta)^n := \mc N(\Lambda,\zeta)\cap \mc N(\zeta)^n,
\end{align*}  as defined in \cite[Section 5]{MS}. 
Then  ${\mc N}(\zeta)$ decomposes into ${\mc N}(\zeta)=\bigoplus_{\Lambda \in \h^\ast /\Upsilon}{\mc N}(\Lambda, \zeta)$. 

Similarly,  we define $\widetilde{\mc N}(\Lambda, \zeta):=\{M\in \widetilde{\mc N}(\zeta)|~\Res M \in\mc N(\Lambda,\zeta) \}$. Since $\mc N(\Lambda,\zeta)$ is stable under tensoring with finite-dimensional $\mf g_\oa$-modules (see \cite[Theorem 4.1, Lemma 4.3]{MS}) and every module in $\mc N(\Lambda,\zeta)$ has finite length, we know that the family of objects in $\widetilde{\mc N}(\Lambda, \zeta)$ consists of composition factors of $\Ind X$, for $X\in {\mc N}(\Lambda, \zeta)$. Therefore  $\Ind$ and $\Res$  restrict to well-defined  functors between $\widetilde{\mc N}(\Lambda, \zeta)$ and ${\mc N}(\Lambda, \zeta)$.  Also, we have  the  decomposition  $\widetilde{\mc N}(\zeta)=\bigoplus_{\Lambda \in \h^\ast /\Upsilon}\widetilde{\mc N}(\Lambda, \zeta) $ (see also \cite[Lemma 4.3]{MS}). 

We define $$\widetilde{\mc N}(\zeta)^n:=\{M\in  \widetilde{\mc N}(\zeta)|~\theta_M(\mf m)^nM=0\}=\{M\in  \widetilde{\mc N}(\zeta)| \Res M \in \mc N(\zeta)^n\},$$  $$\widetilde{\mc N}(\Lambda, \zeta)^n = \{M \in \widetilde{\mc N}(\zeta)^n|~\Res M\in \mc N(\Lambda, \zeta)^n \}.$$ We claim that   $\Ind$ and $\Res$  restrict to well-defined functors between $\widetilde{\mc N}(\zeta)^n$ and ${\mc N}(\zeta)^n$. To see this, let $V\in \mc N(\zeta)^n$ and $s\in \mf m^n$. Consider the canonical epimorphism $U(\g_\ob)\otimes V \onto V$.  By \cite[Theorme 4.3]{MS} we have 
$\theta_{U(\g_\ob)\otimes V}(s)= \text{Id}^{\zeta}_{U(\g_\ob)}\otimes \theta_V(s)=0,$ which implies that $\theta_{\Res\Ind V}(s)=0$, as desired. Therefore  $\Ind$ and $\Res$  restrict to well-defined functors between $\widetilde{\mc N}(\Lambda, \zeta)^n$ and ${\mc N}(\Lambda, \zeta)^n$, for any $n\geq 0$.

\subsection{Harish-Chandra bimodules}
We recall some conventions of Harish-Chandra bimodules; see, e.g., \cite[Section 3]{CC}, or \cite[Section 3]{MS} for more details. In the rest of the present paper, we set $\widetilde{U}:=U(\mf g)$ and $U:=U(\g_\oa).$

\subsubsection{}
We denote by~$\cF$ the category of finite-dimensional semisimple $\fg_{\oa}$-modules. Set  $\widetilde{\mc F}$ to be the category of  finite-dimensional  $\fg$-modules which restrict to objects in~$\cF$.  We denote the full subcategory of projective modules in~$\widetilde{\cF}$ by~$\widetilde{\mc P}$. Modules in $\widetilde{\mc P}$ are precisely the direct summands of modules $\Ind V$, for arbitrary~$V\in\cF$. For a $\g$-module $M$, we denote by $\widetilde{\cF}\otimes M$  the category of 
$\g$-modules of the form $V\otimes M$, with $V\in\widetilde{\cF}$. Similarly, we define $\widetilde{\cP}\otimes M$, ${\cF}\otimes N$ and ${\mc P}\otimes N$, for $N\in \g_\oa$-Mod. 

For a given full subcategory $\mc C$ of either $\g$-Mod or $\g_\oa$-Mod, we denote by $\add(\mc C)$ the category of all modules isomorphic to direct summands 
of objects in $\mc C$. Also, we set $\langle \mc C\rangle$ to be the full subcategory of all modules isomorphic to 
subquotients of modules in~$\mc C$. Let~$\mathrm{Coker}(\widetilde\cF\otimes M)$ denote the {\em coker-category} of $M$ consisting of all ${\g}$-modules $X$ that have a presentation 
$$A\to B \to X\to 0,$$ where 
$A,B\in \add(\widetilde\cF\otimes M)$. 
Similarly we define the coker-category $\mathrm{Coker}(\cF\otimes N)$ of 
$\fg_\oa$-modules $N$ (see  \cite{MaSt08}). 

\subsubsection{} \label{Sect422} For a given  $(\widetilde{U},U)$-bimodule $Y$, we denote by $Y^\ad$  the restriction of $Y$ to the adjoint action of $\g_\oa$. This is the restriction via  $U\hookrightarrow \widetilde{U}\otimes U^{\text{op}},~X\mapsto X\otimes 1-1\otimes X.$ 
 Let $\mc B$ denote the category of  finitely-generated $(\widetilde{U}, U)$-bimodules $N$ for which $N^{\ad}$ is a  direct sum of modules in $\mc F$.  Let $J\subset U$ be a two-sided ideal, denote by $\mc B(J)$ the full subcategory of $\mc B$ consisting of bimodules $N$ such that~$NJ=0$. Also, we set $\mc B_J:= \bigcup_{n\geq 1} \mc B(J^n)$.

\begin{comment}
For any $M\in \mf \g\mod$ (resp. $M\in \mf \g_\oa\mod$), $N\in \mf \g_\oa\mod$, we set $\widetilde{\mc L}(N,M)$ (resp. ${\mc L}(N,M)$) to be the  maximal $(\widetilde{U},U)$-submodule (resp. $(U,U)$-submodule) of~$\Hom_{\mC}(N,M)$ such that~$\Hom_{\mC}(N,M)^{\ad}$ is a  direct sum of modules in $\mc F$. Then we have a canonical monomorphism \begin{align}
&\iota_M:\widetilde{U}/\text{Ann}_{\widetilde{U}}(M)\hookrightarrow \widetilde{\cL}(M,M),~\text{ for }M\in \mf \g\mod,\\
&\iota_M:U/\text{Ann}_{U}(M)\hookrightarrow \cL(M,M),~\text{ for }M\in \mf \g_\oa\mod. \label{eq::KSProb} 
\end{align}
\end{comment}

For a given $\mf g$-module $M$ and a given $\g_\oa$-module $N$, we set ${\mc L}(N,M)$  to be the  maximal $(\widetilde{U},{U})$-submodule  of~$\Hom_{\mC}(N,M)^{\ad}$ that belongs to $\mc B$. Namely,~$\mc L(M,N)$ is  the  maximal submodule which is a  direct sum of modules in $\mc F$ under the usual adjoint action.  We have a canonical monomorphism \begin{align}
&\widetilde{\iota}_M:\widetilde{U}/\text{Ann}_{\widetilde{U}}(M)\hookrightarrow \text{Hom}_{\mathbb C}(M,M)^{\text{ad}},~\text{ for }M\in \mf \g\mod. \label{eq::KSProb}   
%&\iota_N:U/\text{Ann}_{U}(N)\hookrightarrow \text{End}^{\text{adf}}_{\mathbb C}(N,N),~\text{ for }M\in \mf \g_\oa\mod. \label{eq::KSProb} 
\end{align}

The question of surjectivity of the $\widetilde{\iota}_M$ is known as the {\em Kostant's problem}; see   \cite{Jo80,Gor3,MaMe12}.  By slight abuse of notation, we will use the notations $\mc L(-,-)$ and $\iota_{(-)}$ for the corresponding functor applied to the case of  $\mf g=\mf g_\oa$.

%For any $M\in \mf \g\mod$ (resp. $N\in \mf \g_\oa\mod$), we define similar bimodules. Namely, we set $\text{End}^{\text{adf}}_{\mathbb C}(M,M)$ (resp. $\text{End}^{\text{adf}}_{\mathbb C}(N,N)$) to be the  maximal $(\widetilde{U},\widetilde U)$-submodule (resp. $(U,U)$-submodule) of~$\Hom_{\mC}(M,M)$ (resp. $\Hom_{\mC}(N,N)$) which is a  direct sum of modules in $\mc F$ under the usual adjoint action. 

\subsection{Equivalence} \label{Sect43} We are going to establish an equivalence between  $\widetilde{\mc N}(\Lambda, \zeta)$ and a category of Harish-Chandra $(\widetilde{U},U)$-bimodules. The following theorem established in \cite[Theorem~3.1]{CC} is a variation of~\cite[Theorem~3.1]{MS}.

\begin{thm}\label{thm::CCthm31}
	Let $M\in \g_\oa$-Mod.   %$I$:={\emph Ann}_{U}(M)$.
	Set $I$ to be  the annihilator ideal of $M$. Suppose that the monomorphism $\iota_M$ in \eqref{eq::KSProb} is an isomorphism and $M$ is projective in 
	$\langle \cF\otimes M\rangle$. Then 
	$$-\otimes_{U}M\,:\;\;\; \mathcal B(I)\,\to\, \mathrm{Coker}(\widetilde{\cF}\otimes \mathrm{Ind}(M))$$
	is an equivalence of categories with inverse $\mathcal L(M,-)$.
\end{thm}

Theorem \ref{thm::CCthm31} will be the main tool in the proof of Theorem B. Before 
giving the proofs, we need several preparatory results. For a given $\mu\in \h^\ast$, recall that  $\chi_\mu^{\mf l_\zeta}:  Z(\mf l_\zeta)\rightarrow \mathbb C$ denotes the central character of $\mf l_\zeta$ associated with $\mu$. %We  set $I_\mu:=\text{Ker}\chi_\mu^{\mf l_\zeta}$. 

 Recall that $\mf l_\zeta$ is a Levi subalgebra in a parabolic decomposition of $\mf g_\oa$, giving  a corresponding parabolic subalgebra $\mf q_\zeta$ of $\mf g_\oa$. For $\mu \in \h^\ast$, we define \[M^n(\mu, \zeta):= U\otimes_{\mf q_\zeta} Y_\zeta^n(\mu, \zeta),\]  where $Y_\zeta^n(\mu, \zeta):=U(\mf l_\zeta)/(\text{Ker}\chi_\mu^{\mf l_\zeta})^n U(\mf l_\zeta)\otimes_{U(\mf n_\zeta)}\mathbb C_\zeta$; see  \cite[Section 5]{MS}. We put $I_\mu:=U \text{Ker}\chi_\mu^{\oa}$. We set $\mc H(I^n_\mu)$ to be the category of Harish-Chandra $({U},U)$-bimodules $X$ such that $XI_\mu^n=0$. The following lemma is established in \cite[Theorem 5.3]{MS}.
\begin{lem}[Mili{\v{c}}i{\'c}-Soergel] \label{lem::MSthm53}
	Let $\la\in \h^\ast$ be dominant  such that $W_\zeta =W_\la$. Then the functor $N\rightarrow N\otimes_U M^n(\la, \zeta)$ provides an equivalence $T_n$ from  $\mc N(\Lambda, \zeta)^n$ to $\mc H(I_\la^n)$, for each $n\geq 1$.  This  gives rise to an  equivalence $T$ from $\bigcup_{n\geq 1}\mc H(I^n_\la)$ to $\mc N(\Lambda, \zeta)$.
\end{lem}

 Before proving Theorem B, we collect some useful facts from \cite{MS} as follows. 
\begin{lem}[Mili{\v{c}}i{\'c}-Soergel] \label{lem::MSlem}
	Let $\la \in \h^\ast$ be dominant with $W_\la =W_\zeta$. Set $\Lambda:=\la +\Upsilon$. Then for any $n\geq 1$ we have \begin{itemize}
		\item[(1)] $\mc N(\Lambda, \zeta)^n$ is stable under tensoring with  finite-dimensional  $\mf g_\oa$-modules.
		\item[(2)] $\mc N(\Lambda,\zeta)^n$ has enough projective modules, and ${M}^n(\la, \zeta)$ is projective in $\mc N(\zeta)^n$. 
		\item[(3)] ${\emph Ann}_U M^n(\la,\zeta) = I^n_\la$, and $\iota_{M^n(\la,\zeta)}$ is an isomorphism  for $M= M^n(\la,\zeta)$ in \eqref{eq::KSProb}.
		\item[(4)] For any $n>m$ the canonical epimorphism $ M^n(\la,\zeta)\onto M^m(\la,\zeta)$ has kernel {\em $\text{Ker}(\chi_\la^\oa)^m$}$M^n(\la,\zeta)$.
	\end{itemize}
\end{lem}
\begin{proof}
	Part $(1)$ is a consequence of \cite[Theorem 4.1, Lemma 4.3]{MS}. As has been noted in the proof of \cite[Theorem 5.3]{MS}, $\mc H(I_\la^n)$ has enough projective modules. Therefore, conclusions in Part $(2)$ and Part $(3)$ follow from \cite[Lemma 5.11, Proposition 5.5]{MS} and Lemma \ref{lem::MSthm53}. Part (4) is taken from \cite[Lemma 5.14]{MS}.
\end{proof}

\begin{rem}
 The facts $(1)-(3)$ are also given in  the proof of \cite[Theorem 5.3]{MS}.
\end{rem}

We set $\mc B_\la:=\mc B_{I_\la}$ to be the full subcategory consisting of objects $$\{X \in \mc B|~XI_\la^n = 0, \text{ for }n>>0\}.$$  Recall that $\Lambda:=\la+\Upsilon$. We now in a position to state the following equivalence.
\begin{thm} \label{thm::2} Suppose that $\la\in \mf h^\ast$ is dominant such that $W_\zeta =W_\la$.  	Then the functor $X \mapsto \lim\limits_{\leftarrow}X\otimes_U  M^n(\la,\zeta)$ gives an equivalence of categories $\widetilde{T}: \mc B_\la \rightarrow \widetilde{\mc N}(\Lambda, \zeta).$
\end{thm}
\begin{proof}  By definition 
	 $\widetilde{\mc N}(\Lambda, \zeta)^n$ is the full subcategory of $\widetilde{\mc N}(\zeta)$ consisting of modules $M$ such that $\Res M\in \mc N(\Lambda, \zeta)^n$, for any $n\geq 1$.  By \cite[Proposition 2.2.1]{Co}, we know that $\widetilde{\mc N}(\Lambda, \zeta)^n$  has enough projective objects, and any projective module in $\widetilde{\mc N}(\Lambda, \zeta)^n$ is a direct summand of   $\Ind X$,  for  projective object $X\in \mc N(\Lambda, \zeta)^n$. We claim that the functor $ -\otimes_{U} M^n(\la,\zeta)$ gives an equivalence  
$\mathcal B(I^n_\la)\cong \widetilde{\mc N}(\Lambda, \zeta)^n.$

	 Using Part $(2)$ of Lemma \ref{lem::MSlem}, we know  $M^n(\la,\zeta)$ is projective in $\langle \mc F\otimes M^n(\la,\zeta) \rangle \subset {\mc N}(\zeta)^n$. Then  by Part $(3)$ of Lemma \ref{lem::MSlem} and Theorem \ref{thm::CCthm31}, we obtain an  equivalence 
	 \begin{align}
	 &	 -\otimes_{U} M^n(\la,\zeta)\,:\;\;\; \mathcal B(I^n_\la)\,\to\, \mathrm{Coker}(\widetilde{\cF}\otimes \mathrm{Ind}(M^n(\la,\zeta))).
	 \end{align}

To complete the proof, we shall show that $\mathrm{Coker}(\widetilde{\mc F}\otimes \mathrm{Ind}(M^n(\la,\zeta))) = \widetilde{\mc N}(\Lambda ,\zeta)^n$. We may observe that $E\otimes \Ind M^n(\la,\zeta)\cong  \Ind(\Res E\otimes M^n(\la,\zeta))$ is projective in $\widetilde{\mc N}(\Lambda, \zeta)^n$, for any $E\in \widetilde{\mc F}$. Since $\widetilde{\mc N}(\Lambda, \zeta)^n$ has enough projectives, we need just to show that $\mathrm{add}(\widetilde{\mc F}\otimes \mathrm{Ind}(M^n(\la,\zeta)))$ contains all projective modules in $\widetilde{\mc N}(\Lambda,\zeta)^n$. To see this, let $P\in {\mc N}(\Lambda, \zeta)^n$ be a projective object. Since every projective module in $\mc H(I^n_\la)$ is a direct summand of  $E\otimes U/I^n_\la$, for   $E\in \mc F$, we may conclude that $P$ is a direct summand of a projective module of the form $E\otimes M^n(\la,\zeta)$ by Lemma \ref{lem::MSthm53}.  Therefore we have  $\Ind P \in \text{add}(\Ind(\mc F\otimes M^n(\la,\zeta)))$.  Now we calculate   
\begin{align*}
&\widetilde{\mc P}\subseteq \text{add}(\Ind(\mc F\otimes M^n(\la,\zeta))) \\
&\subseteq \text{add}(\Ind({\mc F}\otimes \Res \Ind M^n(\la,\zeta)))\\
&=\text{add}(\Ind\mc F\otimes  \Ind M^n(\la,\zeta)) \\
&= \text{add}(\widetilde{\mc P}\otimes  \Ind M^n(\la,\zeta))\\
&\subseteq \text{add}(\widetilde{\mc F}\otimes  \Ind M^n(\la,\zeta)).
\end{align*}
Consequently, we have equivalence $-\otimes_{U} M^n(\la,\zeta): \mathcal B(I^n_\la) \rightarrow   \widetilde{\mc N}(\Lambda ,\zeta)^n$. By Part (4) of Lemma \ref{lem::MSlem}, the functor $X \mapsto \lim\limits_{\leftarrow}X\otimes_U  M^n(\la,\zeta)$ determines an equivalence of categories $\widetilde{T}: \mc B_\la \rightarrow \widetilde{\mc N}(\Lambda, \zeta).$ This completes the proof. 
\end{proof}

\begin{cor}
Let $\la\in \mf h^\ast$ be dominant and $\zeta,\eta\in \mc I$ such that $W_\la=W_\zeta =W_\eta$. Then we have $$\widetilde{\mc N}( \la +\Upsilon,\eta) \cong \widetilde{\mc N}( \la +\Upsilon,\zeta).$$
\end{cor}

%The following corollary is an analog of \cite[Proposition 4.4]{CC}. 
\begin{cor}	\label{cor::completion}
	Let $\la\in \mf h^\ast$ be dominant such that $W_\la =W_\zeta$. Suppose that $s$ is a simple reflection that $s$ does not lie in the integral Weyl group of $\la$. Then we have $$\widetilde{\mc N}(s\cdot \la +\Upsilon,\eta) \cong \widetilde{\mc N}(\la+\Upsilon,\zeta),$$
\end{cor} for any $\eta \in \mc I$ with $W_{s\cdot \la}=W_{\eta}.$
\begin{proof}
 By Theorem \ref{thm::2}, we have the following equivalences
 \begin{align}
 &\widetilde{\mc N}(s\cdot \la +\Upsilon,\eta) \cong \mc B_{s\cdot \la} = \mc B_\la \cong \widetilde{\mc N}(\la+\Upsilon,\zeta).
 \end{align}  
\end{proof}

\section{Multiplicities of standard Whittaker modules} \label{Sect5}
In this section, we assume that $\mf g =\g_{-1}\oplus \mf g_0\oplus \g_1$ is a classical Lie superalgebra of type I and $\zeta \in \mc I$ such that $\mf l_\zeta$ is a Levi subalgebra in a parabolic decomposition $\mf g=\mf u^-_\zeta\oplus \mf l_\zeta \oplus \mf u_\zeta$ satisfying  
\begin{align}
& (\mf u^-_\zeta)_\ob=\mf g_{-1},~(\mf u_\zeta)_\ob=\mf g_{1} = \mf b_\ob.\label{Sect5eq1}
\end{align}
  In particular, we are mainly interested in the following  concrete subset of the classical Lie superalgebras:
\begin{align}
&\mf g=\gl(m|n),~\mf{osp}(2|2n),~\pn. \label{maintypeI}
\end{align}
In this case,  for an arbitrary $\zeta \in \mc I$, the  $\mf l_\zeta$ is always a Levi subalgebra in a parabolic decomposition  $\mf g=\mf u^-_\zeta\oplus \mf l_\zeta \oplus \mf u_\zeta$ that satisfies \eqref{Sect5eq1}.  %We have thus assumed that $\mf b_\ob =\mf g_1$.
 We refer to  \cite[Section 3]{Mu12} and \cite[Section 5]{CCC} for more details. We note that the definitions of standard Whittaker module $\widetilde{M}(\la, \zeta)$ from \eqref{eq::WhiKac} in Section \ref{sect32} and that from \eqref{eq33} in Section \ref{Sect31} coincide. Also, we remark that the character $\zeta$ on $\mf n$ is naturally extended to a homomorphism from $U(\mf n)$ to $\C$.  %Therefore we will not assume that  $\mf l_\zeta$ is a Levi subalgebra in a parabolic decomposition of $\mf g$. 
%Recall that   $\mf l_\zeta = \mf n_\zeta^-\oplus \mf h\oplus \mf n_\zeta$  denotes the  triangular decomposition of $\mf l_\zeta$. 

\subsection{The category $\mc O$ and the Whittaker vectors} We recall that the BGG category $\mc O$ consists of finitely-generated $\fg$-modules which are semisimple over $\fh$ and locally finite over $\fn$. Therefore $\mc O$ is the category of~$\fg$-modules that restrict by $\Res$ to  $\mf g_\oa$-modules in the BGG category of~\cite{BGG}, which we will denote by $\cO^{\oa}$ in the present paper.  
We refer to \cite{ChWa12,Mu12} for a more complete treatment. In particular, we may note that $\mc O\subset \widetilde{\mc N}(0)$.

Both categories $\mc O$ and $\mc O^\oa$ admit highest category structures (see, e.g., \cite{CPS1,Ma,CCC}). %In the rest of the present paper, we fix a Borel subalgebra $\mf b$ such that $\mf b_\ob =\mf g_1$. 
The partial order $\le_{}$ on $\fh^\ast$ is defined as the transitive closure of the following relations
$$\begin{cases}
\lambda-\alpha \le_{}\lambda, &\mbox{for $\alpha\in \Phi(\fn)$},\\
\lambda+\alpha \le_{}\lambda, &\mbox{for $\alpha\in \Phi(\fn^-)$}.
\end{cases}$$ 
%We use   $\leq $ instead of $\leq_{\mf b}$  when the Borel subalgebra $\mf b$ is clear from the context. 
We define ${M}(\la)$ and $\widetilde{M}(\la)$ to be the Verma modules of highest weight $\la$ as follows
$$M(\la)=U(\mf g_\oa)\otimes_{\fb_{\oa}}\C_\lambda,~\widetilde{M}(\la) := U(\mf g)\otimes_{\mf b}   \C_\la\cong K(M(\la)).$$
Denote by $L(\la)$ and $\widetilde{L}(\la)$ the simple quotients of   $M(\la)$ and $\widetilde{M}(\la)$, respectively. There are canonical epimorphisms $\widetilde{M}(\la)\onto K(L(\la)),~K(L(\la))\onto \widetilde{L}(\la),$ for any $\la \in\h^\ast$.

For any $\mf g$-module $M$ and $\g_\oa$-module $V$, we define $$\Whoa(V):=\{v\in V|~xv=\zeta(x)v,~\text{for any x}\in \mf n_\oa\},$$
$$\Whob(M):=\{m\in M|~xm=\zeta(x)m,~\text{for any x}\in \mf n\}.$$ Following \cite{Ko78},   elements in $\Whoa(V)$ and $\Whob(M)$ are called {\em Whittaker vectors}. The following useful lemma is taken from \cite[Lemma 3.3]{BCW}.
\begin{lem}[Bagci-Christodoulopoulou-Wiesner] \label{lem17}
 The subspace {\em $\Whob(M)$} is non-zero, for any $M\in \widetilde{\mc N}(\zeta)$. In particular, $M$ is simple if $\dim${\em $\Whob(M)$}$=1$.
\end{lem}
\begin{proof} Let $v\in M$ be non-zero. By 
	\cite[Lemma 3.17]{ChWa12} there exists one-dimensional $\mf n$-module $\C_{\eta}$ of $U(\mf n)v$, for some $\eta \in \mc I$. We may conclude that $\eta =\zeta$ since $M\in \widetilde{\mc N}(\zeta).$
\end{proof}
\subsection{The functors $\widetilde{\Gamma}_\zeta$} \label{Sect51}
Let $\zeta \in \mc I. $
For a given   $M \in \mc O$ and $\la \in \h^\ast$, let $M_\la$ denote the corresponding weight subspace, that is, $M_\la : = \{m\in M|~hm=\la(h)m,~\text{for any } h\in \h^\ast\}$.  We may associate the completion $\overline{M}:= \prod_{\la \in \h^\ast} M_\la$, which admits a structure of $\mf g$-module in a natural way. 
 %and define an exact functor from $\mc O$ to $\mf g$-Mod. 
Define $$\WG(M):=\{m\in \overline{M}|~x-\zeta(x) \text{ acts  nilpotently at }m, \text{for any } x\in \mf n_\oa\}.$$  We may note that $\WG(M) \in \g$-Mod since  \begin{align}
&(x-\zeta(x))ym = (\ad(x)y)m+y(x-\zeta(x))m,
\end{align} for any $y\in U(\mf g)$, $x\in \mf n_\oa$ and $m \in \WG(M)$. Therefore $\widetilde{\Gamma}_\zeta$ defines a functor from $\mc O$ to $\mf g$-Mod.

 The $\g$-module $\WG(M)$ is locally finite over $\mf n$ (see, e.g., \cite[Lemma 1]{AB}). We may also note that $\WG(M)$ is the set of elements in $\overline M$ that are annihilated by some power of $\text{Ker}\zeta$, and hence $\Res \WG(M) = \overline\Gamma_\zeta (\Res M) \in \mc N(\zeta)$  (cf.   \cite[Lemma 3.2]{B}), where the functor $\overline \Gamma_\zeta$ is defined in \cite[Section 3.1]{B}. We recall that the functor $\ov \Gamma_\zeta$ sends Verma modules to standard Whittaker modules and sends simple modules to simple modules or zeros by \cite[Proposition 6.9]{B}.

 The conclusion of Theorem C is an immediate consequence of the following theorem:
	\begin{thm} \label{To3rdmainthm} The functor $\widetilde{\Gamma}_\zeta$ defines an  exact functor from $\mc O$ to $\widetilde{\mc N}(\zeta)$. Furthermore, for any $\la \in \h^\ast$ we have 
		\begin{align}
		&\widetilde{\Gamma}_\zeta(\widetilde{M}(\la))=\widetilde{M}(\la,\zeta); \label{qe::1}\\
		&\widetilde{\Gamma}_\zeta(K(L(\la))) \cong \left\{ \begin{array}{ll} K(L(\la, \zeta)),\quad \text{if $\la$ is  $\mf n_\zeta$-antidominant;} \\ 
		0, \quad\text{otherwise};\end{array} \right.  \label{qe::2} \\
		&\widetilde{\Gamma}_\zeta(\widetilde{L}(\la)) \cong \left\{ \begin{array}{ll} \widetilde{L}(\la, \zeta),\quad \text{if $\la$ is  $\mf n_\zeta$-antidominant;} \\ 
		0, \quad\text{otherwise}.\end{array} \right.   \label{qe::3}
		\end{align}
	\end{thm}
	\begin{proof}
		The first claim follows from the isomorphism $\Res\circ \WG  \cong \overline\Gamma_\zeta\circ \Res $ and \cite[Lemma 3.2]{B}. We shall show that $\WG\circ K\cong K\circ \ov{\Gamma}_\zeta$. 
		 For $M\in \mc O^\oa$, the $\mf g$-module  $\WG (K(M))$ can be considered as the set of elements of $K(\overline{M})$ that are annihilated by some powers of $\text{Ker}\zeta$. Therefore we have natural isomorphisms $\WG (K(M))\cong K(\overline \Gamma_\zeta(M))$ by the proof of \cite[Proposition 3]{AB}.  The conclusions of \eqref{qe::1}-\eqref{qe::2} follow from \cite[Proposition 6.9]{B}.
		
		We note that $\widetilde{\Gamma}_\zeta(\widetilde{L}(\la)) =0$  if $\la$ is not $\mf n_\zeta$-antidominant by \eqref{qe::2}.  Now, suppose that  $\la$ is  $\mf n_\zeta$-antidominant. We are going to show that $\widetilde{\Gamma}_\zeta(\widetilde{L}(\la)) \cong \widetilde{L}(\la, \zeta)$. To see this, we first note that there is a non-zero homomorphism $\phi$ from $\widetilde{M}(\la)$ to $\Coind_{\mf h+\mf n^-}^{\g}\C_\la$ since $\la$ is the highest weight in the weight subspace of $\Coind_{\mf h+\mf n^-}^{\g}\C_\la$. Let $S$ denote the image of $\phi$. We note that $\WG(S)$ is a $\mf g$-submodule of $\Coind_{\mf h+\mf n^-}^{\g}\C_\la$ (see also the proof of \cite[Theorem 36]{BM}). Also,  $\WG(S)$ is non-zero since  $\Res \WG(S) \onto  \Res \WG(\widetilde{L}(\la))\cong \overline \Gamma_\zeta(\Res \widetilde L(\la))\supset  \ov{\Gamma}_\zeta(L(\la)) =L(\la, \zeta)$ by \cite[Proposition 6.9]{B}. 
		
	 %	Since $\WG(S)$ has finite length, 
	 	Since $\WG(S)$ is non-zero, the subspace $\Whob(\WG(S))$ of Whittaker vectors of $\WG(S)$ is non-zero. By a similar argument as used in \cite[Lemma 37]{BM}, we also know that the subspace $\Whob({\Coind_{\mf h+\mf n^-}^{\g}\C_\la})$ of Whittaker vectors of $\Coind_{\mf h+\mf n^-}^{\g}\C_\la$ is of one-dimensional, which implies that  $\WG(S)$ is simple.  Since  $\WG(S)$ is the simple quotient of $\WG(\widetilde{M}(\la)) \cong \widetilde{M}(\la,\zeta)$, we may conclude that $\WG(S)\cong \widetilde{L}(\la ,\zeta)$. Also, since $\WG(\widetilde L(\la))$ is a non-zero quotient of $\WG(S)$ by the exactness of $\widetilde{\Gamma}_\zeta$, we have $\WG(\widetilde L(\la))\cong \widetilde{L}(\la, \zeta)$. This completes the proof. 
	\end{proof}

	The following corollary is an analog of Kostant's characterizations of simple Whittaker modules in \cite[Theorem 3.6.1]{Ko78}.
	\begin{cor} \label{Cor19}
	 %	Suppose that $\mf g$ is a classical Lie superalgebra of type I.
	  Let $M\in \widetilde{\mc N}$. Then $M$ is simple if and only if $\dim${\em $\Whob(M)$}$=1$. %of  Whittaker vectors in $M$ is of one-dimensional. 
	\end{cor}
\begin{proof}
  By the proof of Theorem \ref{To3rdmainthm}, we have $\dim\Whob(\widetilde{L}(\la,\zeta)) =1,$ for any $\la \in \h^\ast.$ The conclusion follows from Lemma \ref{lem17}.
\end{proof}

\subsection{Basic Lie superalgebras of type I} \label{Sect52} The series of types A and C Lie superalgebras from the list \eqref{Kaclist1}, \eqref{Kaclist2} belong to the so-called series of {\em basic} Lie superalgebras; see \cite[Section 1.1]{ChWa12}.  
In this subsection, we will give a detailed example of the composition series for the standard Whittaker modules over $\gl(1|2)$.  We will also study several criteria of simplicity and the annihilator for   standard Whittaker modules over Lie superalgebras of types A,  C and  P.  Following \cite{Ko78,MS}, an element $\zeta \in \mc I$ is called {\em regular} (or {\em nonsingular}) if $W_\zeta=W$. 

\subsubsection{The general linear  Lie superalgebras $\gl(m|n)$}
For positive integers $m,n$, the general linear Lie superalgebra $\mathfrak{gl}(m|n)$ 
can be realized as the space of $(m+n) \times (m+n)$ complex matrices
\begin{align} \label{glreal}
\left( \begin{array}{cc} A & B\\
C & D\\
\end{array} \right),
\end{align}
where $A,B,C$ and $D$ are $m\times m, m\times n, n\times m, n\times n$ matrices,
respectively. The bracket is given by the super commutator.  For any $1\leq a,b \leq m+n$, set  $E_{ab}$ to be the elementary matrix in $\mathfrak{gl}(m|n)$, namely, the $(a,b)$-entry of $E_{ab}$  is equal to $1$ and all other entries are $0$.

The Cartan subalgebra $\mf h \subset \mf g_\oo$ consists of diagonal matrices above. We denote the dual basis of $\mf h^*$ by $\{\vare_1, \vare_2, \ldots,\vare_{m+n}\}$ with respect to the following standard basis of $\mf h$ 
\begin{align}
\{H_i:=E_{i,i}|~1\leq i \leq m+n \}. \label{eq::glcartan}
\end{align} The space $\mf h^\ast$ is equipped with a natural bilinear form $(\cdot,\cdot):\mf h^\ast \times \mf h^\ast \rightarrow \C$ by letting $(\vare_i,\vare_j)=\delta_{ij}.$  We fix a triangular decomposition $\mf g=\mf n^-\oplus \mf h\oplus \mf n$, where $\mf n$ and $\mf n^-$ consisting of all strict upper and lower triangular matrices in  \eqref{glreal}, respectively. The corresponding Borel subalgebra is $\mf b = \mf h \oplus \mf n.$  

Recall that we denote by  $\Phi$ the set  of roots and by $\Phi^+$ the set  of positive roots  in the Borel subalgebra $\mf b$. Let $\Phi_\oa$ and $\Phi_\ob$ be the   sets of even  and odd roots in $\Phi$, respectively. The {\em Weyl vector} $\rho$ is defined as 
\[\rho= \frac{1}{2}\sum_{\alpha\in \Phi_\oa^+}\alpha -\frac{1}{2}\sum_{\beta\in \Phi_\ob^+}\beta,\]
where $\Phi_i^+:= \Phi_i\cap \Phi^+,$ for $i=\oa,\ob.$ We recall that a weight $\la$ is {\em typical} if $(\la+\rho,\alpha)\neq 0$, for any $\alpha \in \Phi_\ob$ and is {atypical} otherwise; see, e.g., \cite[Section 2.2.6]{ChWa12}.

\subsubsection{Example: $\mf g=\gl(1|2)$} \label{Sect522ex}
		We now  consider $\mf g=\gl(1|2)$.	The sets $\Phi_\oa^+$, $\Phi_\ob^+$ are given as follows:
		\begin{align}\label{eqroots}
		&\Phi_\oa^+=\{\epsilon_2-\epsilon_3\},~\Phi^+_\ob=\{\vare_1-\vare_2,~\vare_1-\vare_3\}.
		\end{align} 
 %	We define $\Phi^-_\ob=-\Phi^+_\ob$.  
 Also, we set $F_{ij}:=E_{ij}$, for $1\leq j<i\leq 3$.  Let $\zeta \in \mc I$. Note that   standard Whittaker modules and Verma modules coincide in the case when $\zeta=0$. Throughout this subsection, we let $\zeta$ be regular, namely, $\zeta(E_{23})\neq 0$. In this case we have $M(\la ,\zeta) =L(\la, \zeta)$ for any $\la \in \h^\ast$ by \cite[Theorem 3.6.1]{Ko78}.

		In this subsection, we will  construct composition series of standard Whittaker modules of $\widetilde{\mc N}(\zeta)$ explicitly by finding their Whittaker vectors. Similar computation was also given in \cite[Section 5.1]{BCW}, where the authors concluded that all standard Whittaker modules are simple. However, by \cite[Theorem 6.7]{CM} and \cite[Proposition 2.1(3)]{MS} there exist reducible standard Whittaker modules; see Section  \ref{Section523} for more details. It is worth pointing out that the assumption of typical weight $\la$ is needed to be added in the calculation in \cite[Section 5.1]{BCW}. %We will re-compute this example without using Theorem B in this subsection. 
		
		We define the Chevalley generators $f:=E_{32},~e:=E_{23}$ and $h:=E_{22}-E_{33}$ for $[\g_\oa,\g_\oa]\cong \mf{sl}(2)$. The Casimir element $\Omega$ of $\mf g_\oa$ is given by $\Omega:= 4fe+h^2+2h$. Let $z:= E_{11}+\frac{1}{2}(E_{22}+E_{33})\in Z(\mf g_\oa)$. For $\la \in \mf h^\ast$, let $\chi^\oa_\la: Z(\mf g_\oa) \rightarrow \C$ be the central character associated with $\la$. We set $a:=\zeta(e)\neq 0$,   $b:=\chi^\oa_\la(\Omega)$ and $c:=\chi^\oa_\la(z)$. Let $\la =\la_1\vare_1+\la_2\vare_2+\la_3\vare_3 \in \h^\ast$ with complex numbers $\la_i's.$ Recall that we have defined $\widetilde{U}:=U(\g),~{U}:=U(\g_\oa)$.
		
		The following lemma is a consequence of \cite[Corollary 6.8]{CM} and \cite[Proposition 2.1(3)]{MS}
		\begin{lem} \label{lem::18} Consider $\g=\gl(1|2)$ with notations as above. 
	     The following are equivalent:
	     \begin{itemize}
	     	\item[(1)] $\la$ is atypical.
	     	\item[(2)] $b=4(c^2-c).$
	     	\item[(3)] $\widetilde{M}(\la, \zeta)$ is not simple. 
	     \end{itemize}   
		\end{lem}
	\begin{proof} By \cite[Theorem 3.9]{Ko78} (see also \cite[Proposition 2.1(3)]{MS}) we know that  the annihilators of   $\mf g_\oa$-modules $M(\la, \zeta)$ and $M(\la)$ coincide.  The fact that  Part (1) and Part (3) are indeed equivalent was established in  \cite[Corollary 6.8]{CM}. 
		
		The equality $b=4c(c-1)$ holds if and only if
		\begin{align*}
		&(\la_2-\la_3)(\la_2-\la_3+2) = 4(\la_1+\frac{1}{2}\la_2+\frac{1}{2}\la_3)
		(\la_1+\frac{1}{2}\la_2+\frac{1}{2}\la_3-1),
		\end{align*} which is equivalent to $(\la_1+\la_2)(\la_1+\la_3-1)=0$ by a direct computation. This shows that Part (1) and Part (2) are equivalent. The conclusion follows. 
	\end{proof}
\begin{proof}[Alternative proof of Lemma \ref{lem::18}] By a direct computation we have $$E_{12}E_{13}F_{31}F_{21}x=(z^2 -z -\frac{1}{4}\Omega) x= (c^2-c-\frac{1}{4}b)x,$$ for any $x\in M(\la, \zeta)$ (see also \cite[Example 6.6]{CM}). The conclusion follows from \cite[Theorem 6.7]{CM}. 
\end{proof}

		We will generalize Lemma \ref{lem::18} later; see  Theorem \ref{thm23}. We now turn to the composition series of $\widetilde{M}(\la, \zeta)$ for atypical weight $\la$.  The following lemma will be  useful.

		\begin{lem} \label{lem19}
			For each atypical weight $\la\in \h^\ast$,  there are exactly two antidominant composition factors of $\widetilde{M}(\la)$. They are
		 $\widetilde{L}(\underline{\la}),~ \widetilde{L}(\underline{\la} -\alpha),$ where $\underline{\la}\in W\cdot \la$ is antidominant and $\alpha$ is the unique positive  odd root such that $(\underline{\la}+\rho, \alpha)=0$ and $\underline{\la} -\alpha$ is antidominant.
		\end{lem}
	\begin{proof}
		Suppose that $\la$ is non-integral. In this case, it is known that the block $\mc O_\la$ is equivalent to the principal block of $\gl(1|1)$ (see, e.g., \cite[Section 4.2]{CCL2} for an argument), and every simple module in $\mc O_\la$ are antidominant. Therefore   $\widetilde{M}(\la)$ has exactly two desired composition factors $L(\underline \la)$ and $L(\underline \la -\alpha)$ that are all antidominant. 
		
		Suppose that $\la$ is integral. Then the conclusion follows from the BGG reciprocity and Lemma \ref{lemapp1} in Section \ref{Sectapp}.  
	\end{proof}

		The $\mf g_\oa$-module $M(\la, \zeta)$ can be regarded as a submodule of $\Res\widetilde{M}(\la ,\zeta)$. Let $v\in \Whoa( M(\la,\zeta))$ be a non-zero Whittaker vector. By \cite[Lemma 5.6]{BCW}, the set 
		\begin{align*}
		&\{v_1:=v,~v_2:=F_{21}v,~v_3:=F_{21}F_{31}v,~v_4:=2aF_{31}v+F_{21}hv\},
		\end{align*} forms a basis for $\Whoa(\widetilde{M}(\la, \zeta)).$ 
		
	%	The following useful lemma was given in \cite[Section 6.2.2]{CM} 
	%	\begin{lem} Let $X$ be a simple $\mf g_\oa$-module. Then $K(X)$ is simple if and only if 	the central element $z^2 -z -(1/4)\Omega$ acts on $X$ as a non-zero scalar.  		\end{lem}
	%	\begin{proof}	By a direct computation we have $E_{12}E_{13}F_{31}F_{21}x=(z^2 -z -(1/4))\Omega x$, for any $x\in X$. The conclusion follows from \cite[Theorem 6.7]{CM}. \end{proof}

		\begin{prop} \label{pro20}
			Suppose that $\la$ is atypical with  $\underline{\la}\in W\cdot \la$ antidominant.   Let $\alpha$ be the unique positive odd root such that $(\underline{\la}+\rho, \alpha)=0$ and  $\underline{\la} -\alpha$ is antidominant. Then there is a short exact sequence
				\begin{align}
			    &0 \rightarrow  \widetilde{U}w\rightarrow \widetilde{M}(\la, \zeta)\rightarrow \widetilde{L}(\la,\zeta) \rightarrow 0,
				\end{align}
				where the submodule $\widetilde{U}w\cong \widetilde{L}(\underline{\la}-\alpha, \zeta)$ is  generated by the Whittaker vector 
				\begin{align}
				&w= \left\{\begin{array}{ll}
				v_2+\frac{1}{2(1-c)}v_4, &  \text{~for } c\neq 1;\\
				%v_2+\frac{1}{2}v_4, & \text{for } c=0;\\
			v_4, & \text{for } c=1;
				\end{array} \right.  
				\end{align} and $ \widetilde{L}(\la, \zeta)$ generated by the image of $v$ in the quotient  $\widetilde{M}(\la, \zeta)/\widetilde{U}w$.

				\begin{comment}
				\item[(2)]
			 If $c=0$. Then there is a short exact sequence
			\begin{align}
			&0 \rightarrow X\rightarrow \widetilde{M}(\la, f)\rightarrow Y \rightarrow 0,
			\end{align}
			where $X\cong \widetilde{L}(\la-\alpha, f)$ generated by the Whittaker vector 
			\begin{align}
			&w=v_2+\frac{1}{2}v_4, 
			\end{align} and $Y\cong \widetilde{L}(\la, f)$ generated by the image of $v$ in the quotient  $\widetilde{M}(\la, f)/X$. 
			
				\item[(3)]
			If $c=1$. Then there is a short exact sequence
			\begin{align}
			&0 \rightarrow X\rightarrow \widetilde{M}(\la, f)\rightarrow Y \rightarrow 0,
			\end{align}
			where $X\cong \widetilde{L}(\la-\alpha, f)$ generated by the Whittaker vector 
			\begin{align}
			&w=v_4, 
			\end{align} and $Y\cong \widetilde{L}(\la, f)$ generated by the image of $v$ in the quotient  $\widetilde{M}(\la, f)/X$. 	\end{comment}

		\end{prop}
		\begin{proof}
		%	We note  \begin{align*} &E_{12}v_1=0,~E_{12}v_2 =\ov H_{12}v, \end{align*}
     Set $w:= Bv_2+Cv_3+Dv_4$ with $B,C,D\in \C$. By a direct computation  we have 
     \begin{align*}
     &E_{12}w=0\Longleftrightarrow Bc=\frac{-1}{2}Db, ~B =2D(1-c),~C=0;\\ 
     &E_{13}w=0 \Longleftrightarrow B =2D(1-c),~C=0.
     \end{align*}  
    % \begin{align}&b =4(c^2 -c). \label{eq::value}  \end{align} 
    
    Consequently, the equality $b =4(c^2 -c)$ coming from Lemma  \ref{lem::18} determines all relations between the coefficients $B$, $C$ and $D$ such that  $w\in \Whob(\widetilde{M}(\la, \zeta))$: 
    	\begin{align}
    &\left\{\begin{array}{ll} \label{eq512}
    B=2(1-c)D, &  \text{~for } c\neq 1;\\
   % B=2D, & \text{for } c=0;\\
    B=0~\text{ and $D$ is arbitrary,} & \text{for } c=1;\\
    C=0.
    \end{array} \right.  
    \end{align}
    
    By Lemma \ref{lem19} and Theorem B,  the   vector $w\in  \Whob(\widetilde{M}(\la ,\zeta))$ satisfying \eqref{eq512}   generates the desired proper simple submodule $\widetilde{U}w\cong \widetilde{L}(\underline \la -\alpha, \zeta).$ This completes the proof. 	\end{proof}

We also give an alternative proof of Proposition \ref{pro20}  without using Theorem B and the character formulas of $\gl(2|1)$ in Section \ref{Sectapp}.
 
\begin{proof}[Alternative proof of Proposition \ref{pro20}]	
	We claim that the length of a composition series of $\Res \widetilde{L}(\la, \zeta)$ is always $2$,  for any atypical weight $\la \in \h^\ast$. To see this, we first note that the length of $\Res\widetilde{M}(\la, \zeta)\cong \Lambda{\g_{-1}}\otimes{L}(\la, \zeta)$ is the dimension $\dim \Lambda{\g_{-1}} =4$ by \cite[Proposition 5.1]{BCW} (see also \cite[Theorem 4.6]{Ko78}). By Lemma \ref{lem::18}, it suffices to show that the length of $\Res \widetilde{L}(\la, \zeta)$ is equal to  or greater than $2$.
	
	By \cite[Theorem 4.1]{CM} there exists $\mu\in \h^\ast$ such that $\widetilde{L}(\la, \zeta)$ is the socle of $\widetilde{M}(\mu, \zeta)$, which is generated by $\Lambda^{\text{top}}\mf g_{-1}\otimes L(\mu,\zeta)$. If $\Res \widetilde{L}(\la, \zeta)$ is of length one then $\Res\widetilde{L}(\la, \zeta) = \Lambda^{\text{top}}\mf g_{-1} \otimes L(\mu, \zeta) \subset \Res \widetilde{M}(\mu, \zeta)$. But for any nonzero $v\in\Whoa(L(\mu,\zeta))$ we calculate \[E_{12}F_{21}F_{31}v = F_{31}(\frac{1}{2}h-1)v-\frac{1}{4a}F_{21}((b-2h-h^2)v),\]
	which is nonzero by \cite[Lemma 2]{Mc2} (see also  \cite[Lemma 5.5]{BCW}). Therefore we have shown that the length of $\Res \widetilde{L}(\la, \zeta)$ is  $2$. Consequently, there is a short exact sequence
	\begin{align}
	&0 \rightarrow X\rightarrow \widetilde{M}(\la, \zeta)\rightarrow Y \rightarrow 0,
	\end{align} where both $X,Y$ are simple Whittaker  modules such that $X$ is generated by the Whittaker vector $w$ from \eqref{eq512}.

	If $c\neq 0,1$ and $(\la+\rho, \beta)=0$, for some positive odd root $\beta$. Then by a direct computation we have  $$\Omega w= (4c^2-8c+3 )w,~zw = (\frac{-1}{2}+c)w,$$ which implies that $Uw$ admits the central character $\chi^\oa_{\la -\beta}$ associated with the weight $\la-\beta$. Since $\widetilde{U}w$ is a proper simple submodule, we may conclude that $\widetilde{U}w\cong \widetilde{L}(\la -\beta, \zeta)\cong \widetilde{L}(\underline \la -\alpha, \zeta)$, as desired. The remaining cases $c=0,1$ can be calculated by similar arguments. The conclusion follows.  
\end{proof}

The following corollary gives a description of block decomposition of $\widetilde{\mc N}(\zeta)$ for $\g=\gl(1|2).$
\begin{cor}
	Consider $\mf g=\gl(1|2)$. Let $\la,\mu \in \h^\ast$ be atypical and $\zeta \in \mc I$ regular. Then $\widetilde{L}(\la, \zeta)$ and $\widetilde{L}(\mu, \zeta)$ lie in the same indecomposable block of $\widetilde{\mc N}(\zeta)$ if and only if $\la \in W\cdot (\mu +k\alpha),$ where $\alpha$ is an odd root with $(\mu+\rho, \alpha)=0$.
\end{cor}

\subsubsection{Criteria for simplicity of standard Whittaker modules} \label{Section523}

In this subsection, we study  the simplicity of standard Whittaker modules for classical Lie superalgebras of types  A and  C from \eqref{eq::claA}, \eqref{eq::claC}.  We refer to \cite{ChWa12} for more details about the ortho-symplectic Lie superalgebras $\mf{osp}(m|2n)$. In particular, the notions of typical and atypical weights for    $\mf{osp}(m|2n)$ are defined in a similar fashion; see \cite[Section 2.2.6]{ChWa12}.

 For any $M\in \g$-Mod, denote by $\text{Ann}_{\widetilde{U}}M$ the annihilator of $M$. Similarly, we define the annihilator  $\text{Ann}_{U(\mf k)}N$ for any subalgebra $\mf k\subseteq \g$ and $N\in \mf k$-Mod.  We first show that the annihilators of standard Whittaker modules and Verma modules coincide.

\begin{prop} \label{prop23} Let $\mf g$ be a classical Lie superalgebra of type I. Then $${\emph Ann}_{\widetilde{U}}\widetilde{M}(\la, \zeta) = {\emph Ann}_{\widetilde U}\widetilde{M}(\la).$$
	
	In particular, if $\mf g$ is basic  then ${\emph Ann}_{\widetilde{U}}\widetilde{M}(\la, \zeta)$ is centrally generated for typical $\la$. 
\end{prop} 
\begin{proof} Let $M^{\mf l_\zeta}(\la)$ denote the parabolic Verma module over $\mf l_\zeta$ of highest weight $\la$. By \cite[Theorem 3.9]{Ko78}, $\text{Ann}_{U(\mf l_\zeta)}Y_\zeta(\la,\zeta) = \text{Ann}_{U(\mf l_\zeta)} M^{\mf l_\zeta}(\la).$ The conclusion follows by an argument similar to the one in the proof of \cite[Proposition 5.1.7]{Di}  (see also \cite[Lemma 2.2]{MS} and \cite[Lemma 6.4]{CM}). 
\end{proof}

In the rest of this subsection, we let $\g$ be one of the series of   $\gl(m|n)$ and  $\mf{osp}(2|2n)$. The following lemma is taken from \cite[Corollary 6.8]{CM}.
\begin{lem} \label{lem::CMCor68}
	Let $V$ be a simple $\mf g_\oa$-module. Then $K(V)$ is simple if and only if ${\emph Ann}_{U} V ={\emph Ann}_{U} L(\la)$, for some typical $\la \in \h^\ast$.  
\end{lem}

%We recall that the multiplicities of standard Whittaker modules have been given in \cite[Theorem 6.2]{B}.
%\begin{thm}[Backelin]\end{thm}

%\red{We recall the simplicity of $M(\la, f)$ has been given in \cite[Theorem 6.2]{B}. That is, $M(\la, f)$ is simple if and only if ..........., for any $\mu \in W\cdot \la.$ We obtain the criteria of the simplicity of standard Whittaker modules.}

\begin{prop} \label{thm23} For any $\la \in \h^\ast$ and $\zeta \in \mc I$, the following are equivalent. 
	\begin{itemize}
		\item[(1)] $\widetilde{M}(\la, \zeta)$ is simple.
		\item[(2)] $\la$ is typical and $M(\la, \zeta)$  is simple.
		\item[(3)] $\la$ is typical and there is a unique $\mf n_\zeta$-antidominant weight $\nu\in W\cdot \la$ with $\la \geq\nu.$ 
	\end{itemize} 
	
	In particular, if $\zeta$ is regular then $\widetilde{M}(\la, \zeta)$ is simple if and only if $\la$ is typical. 
\end{prop}
\begin{proof} The fact that Part $(2)$ and Part $(3)$ are indeed equivalent is a consequence of \cite[Theorem 6.2]{B}.  Next, we recall that the socle $\text{soc} M(\la)$ of $M(\la)$ has a typical highest weight if and only if $\la$ is typical.    By \cite[Proposition 2.1(3)]{MS} it follows that  
	\begin{align} 
	&\text{Ann}_{U} M(\la, \zeta) =\text{Ann}_{U} M(\la) = \text{Ann}_{U} \text{soc} M(\la). \label{eq::Ann1}
	\end{align} 
	
 Now, suppose that $\widetilde{M}(\la, \zeta)$ is simple. Then the simplicity of $M(\la, \zeta)$ follows from the exactness of Kac functor $K(-)$ and  \eqref{eq::WhiKac}. By Lemma \ref{lem::CMCor68}, the socle $\text{soc} M(\la)$ is a simple $\mf g_\oa$-module of typical highest weight, and so  $\la$ is typical.   This shows the implication $(1)\Rightarrow (2)$.

 Conversely, suppose that $M(\la, \zeta)$ is simple and $\la$ is typical.  Then the proof of direction $(2)\Rightarrow (1)$ follows by Lemma \ref{lem::CMCor68}. 
 Finally, by \cite[Theorem 3.6.1]{Ko78}, we know that $M(\la,\zeta)$ is simple if $\zeta$ is regular. This completes the proof.
\end{proof}

\begin{proof}[Alternative proof of Propotision \ref{thm23}] As has been mentioned, Parts $(2)$ and $(3)$ are equivalent by \cite[Theorem 6.2]{B}. It remains to show $(1)\Leftrightarrow (2)$. First, we suppose on the contrary that $\widetilde{M}(\la, \zeta)$ is simple with $\la$ atypical. We know that ${M}(\la, \zeta)$ is simple by the exactness of the functor $K(-)$.  Let $\underline{\la}\in W\cdot \la$ be  antidominant. By  \cite[Theorem 51]{CCM} and Lemma  \ref{lem::CMCor68} the socle of $M(\la)$ is isomorphic to the socle of $K(\underline{\la})$, which   is a simple module of antidominant highest weight $\gamma$ with $\underline{\la} \neq \gamma$; see also \cite[Theorem 4.4, Proposition 4.15]{CCC}. Using the grading operator $d^\g$ from \cite[Sections 5.1, 5.2]{CM}, we know $\gamma \not \in W\cdot \la$, and so $\gamma \not \in W_\zeta\cdot \la$.  By Theorem \ref{maintypeI}, $\WG(\widetilde{L}(\gamma))\cong \widetilde{L}(\gamma,\zeta)$ is a composition factor that is not isomorphic to $\widetilde{L}(\la, \zeta)$ by Theorem \ref{mainthm1typeI}, a contradiction. 
	
	Conversely, suppose  that $\la$ is typical. Then we have $[\widetilde{M}(\la):\widetilde{L}(\mu)] =[M(\la):L(\mu)]$, for any $\mu\in \h^\ast$ (see, e.g., \cite[Theorem 1.3.1]{Gor2}). Therefore the simplicities of ${M}(\la, \zeta)$ and $\widetilde{M}(\la, \zeta)$ are equivalent by Theorem B. 
\end{proof}

\begin{cor}
Let $\la\in \h^\ast$ and $\zeta\in \mc I$ satisfy one of conditions  (1)-(3). Then $${\emph Ann}_{\widetilde{U}}\widetilde{L}(\la, \zeta) = {\emph Ann}_{\widetilde U}\widetilde{L}(\underline{\la}),$$ is centrally generated,  
where $\underline{\la}\in W\cdot \la$ is antidominant. 
\end{cor}

%\subsubsection{Indecomposable blocks of $\mc N(f)$ for $\mf g=\gl(m|n)$ and regular $f$}

\subsection{The periplectic Lie superalgebras} \label{Sect53}
The standard matrix realization of the periplectic Lie superalgebra 
$\pn$ is given by
\begin{align}\label{plrealization}
 \mf{pe}(n):=
\left\{ \left( \begin{array}{cc} A & B\\
C & -A^t\\
\end{array} \right)\| ~ A,B,C\in \C^{n\times n},~\text{$B^t=B$ and $C^t=-C$} \right\} \subset \mathfrak{gl}(n|n).
\end{align}
 The Cartan subalgebra $\mf h \subset \mf g_\oo$ consists  of diagonal matrices above. 
 There is a standard basis of $\mf h$ defined as
\begin{align}
\{H_i:=E_{i,i}-E_{n+i,n+i}|~1\leq i \leq n \}, \label{eq::cartan}
\end{align}  
where $E_{a,b}\in \gl(n|n)$ denotes the $(a,b)$-matrix unit, 
for $1\leq a,b \leq 2n$. We denote  by $\{\vare_1, \vare_2, \ldots,\vare_{n}\}$ the dual basis for $\mf h^*$ with respect to the basis $\{H_i|~1\leq i\leq n\}$. 
 
 We recall from \cite[Section 5]{Se02} that a weight 
 $\displaystyle\la =\sum_{1\leq i\leq n}\la_i\vare_i\in \h^\ast$ is called {\em typical}  if  
 \[  \prod_{1\leq i\neq j \leq  n}(\la_i -\la_j +j-i-1) \neq 0. \]
  
  The following proposition gives composition factors of the standard Whittaker module $\widetilde{M}(\la, \zeta)$ in terns of the Kazhdan-Lusztig combinatorics for typical weight  $\la$. 
 \begin{prop}
 Let $\la \in \h^\ast$ be typical. Then for any  $\zeta\in \mc I$ we have   
 \begin{align}
 &[\widetilde{M}(\la, \zeta): \widetilde{L}(\mu, \zeta)] = [{M}(\la): {L}(\nu)],
 \end{align} where $\nu\in W_\zeta\cdot \mu$ is $\mf n_\zeta$-antidominant. 
 \end{prop}
\begin{proof}
 By \cite[Corollary 4.4]{CP} we have  $[\widetilde{M}(\la):\widetilde{L}(\mu)]=[{M}(\la):{L}(\mu)]$, for any $\mu\in \h^\ast$ (ass also \cite[Corollary 5.8]{Se02}). The conclusion follows Theorem \ref{To3rdmainthm}. 
\end{proof}

We have the following sufficient condition for the simplicity of standard Whittaker modules over $\pn.$
 
 %Then $\widetilde{M}(\la ,f)$ is simple if and only if there is a unique composition factor of $\widetilde{M}(\la)$ that is $\mf n^f_\oa$-antidominant. 
 
	\begin{cor} Let $\la \in \h^\ast$ and $\zeta \in \mc I$. Then $\widetilde{M}(\la ,\zeta)=\widetilde{L}(\la ,\zeta)$ if either $\la$ is antidominant or $\zeta\in \mc I$ is regular. In particular, if $\la$ is antidominant then we have $${\emph Ann}_{\widetilde{U}}\widetilde{L}(\la, \zeta) = {\emph Ann}_{\widetilde U}\widetilde{L}(\la).$$ %In particular, in this case we have  {\em $\Whob(\widetilde{M}(\la, \zeta))$} $=$ {\em $\Whoa({M}(\la, \zeta))$}.
	\end{cor}
	\begin{proof} 
		It was shown in \cite[Lemma 5.11]{CC} (see also  \cite[Lemma 3.2]{Se02}) that $\widetilde{M}(\la)$ is simple if $\la$ is  antidominant. The conclusion follows from Theorem  \ref{To3rdmainthm} and Proposition \ref{prop23}.
	\end{proof}

	\begin{ex}
		Consider $\mf g=\mf{pe}(2)$. Set $X_{12}:=E_{14}+E_{23},~Y_{12}:=E_{32}-E_{41}$. Suppose that $\zeta\in \mc I$ is regular. Then $M(\la, \zeta)$ is simple for any $\la \in \h^\ast$. We have $\Whoa(\Res \widetilde{M}(\la, \zeta))= \C v \oplus \C Y_{12}v$, for any nonzero vector $v\in \Whoa(M(\la, \zeta))$. Then the elements in the set $\{(H_{1}-H_{2})^kv|~k\geq 0\}$ are linear independent by \cite[Lemma 2]{Mc2}. By a direct computation we obtain that $X_{12}Y_{12}v= (H_{1}-H_{2})v$, and so this verifies that  $\Whob(\widetilde{M}(\la, \zeta)) =\C v$.
	\end{ex}

\begin{comment}
\begin{prop}
	Let $\la,f\in \mf h^\ast$. Suppose that $\la$ is typical and $f$ is regular. Then $\widetilde{M}(\la,f)$ is simple. 
\end{prop}
\begin{proof} Recall that $\Ind_{\mf g_\oa+\g_1}^\g (V)$ is simple if $\text{Ann}_U V = \text{Ann}_U L(\la)$, for some typical $\la \in \h^\ast$. Since $f$ is regular, $M(\la, f)$ is a simple $\mf g_\oa$-module 	by \cite[Theorem 3.6.1]{Ko78}. The conclusion follows from \eqref{eq::Ann1}.
\end{proof}\end{comment}

\subsection{The equivalence $\widetilde{T}$} \label{sect55}

In this subsection, we continue to work under the assumption that  $\g$ is a classical Lie superalgebra of type I. Let  $\la\in \h^\ast$ and $\zeta\in \mc I$   such that $\la$ is dominant with $W_\la =W_\zeta$. 

We recall that the equivalence $\widetilde{T}$ of the categories $\mc B_\la$ and $\widetilde{\mc N}(\la+\Upsilon,  \zeta)$ as  constructed in Theorem \ref{thm::2}. We are going to consider the effect of $\widetilde{T}$ on standard   and simple objects in $\mc B_\la$. The following is an analog of \cite[Proposition 5.15]{MS}. 
\begin{prop}  \label{prop32}  For any $\mu \in \la+\Upsilon$, the module $\mc L(M(\la), \widetilde{M}(\mu))$ has a simple top $S_{\la,\zeta}(\mu)$ such that 	
	 \begin{align}
	 &\widetilde{T}(\mc L(M(\la), \widetilde{M}(\mu))) = \widetilde{M}(\mu, \zeta), \label{eq261}\\
	 &\widetilde{T}(S_{\la,\zeta}(\mu)) = \widetilde{L}(\mu, \zeta). \label{eq262} 
	 \end{align}
\end{prop}
\begin{proof}
 We first note that 
	 \begin{align*}
	&\widetilde{T}(\mc L(M(\la), \widetilde{M}(\mu))) \cong  \widetilde{T}(\mc L(M(\la), K({M}(\mu))))  \cong  K({T}(\mc L(M(\la), {M}(\mu)))),
	\end{align*} where  $T$ is the equivalence   in Lemma \ref{lem::MSthm53}.  By \cite[Proposition 5.15]{MS},  we know that ${T}(\mc L(M(\la), {M}(\mu)))\cong M(\mu, \zeta)$. This claim follows.
	\end{proof}

 Now, we put together all the pieces from previous sections. Then, Theorem \ref{thm::2}, Theorem \ref{To3rdmainthm} and   Proposition \ref{prop32} imply that  the composition factors of the {\em standard  Harish-Chandra bimodule} $\mc L(M(\la), \widetilde{M}(\mu))$ in Proposition \ref{prop32} can be  computed via Kazhdan-Lusztig combinatorics   for $\g=\gl(m|n)$ and $\mf{osp}(2|2n)$.

\section{Appendix} \label{Sectapp}
\subsection{Character formulae of $\gl(2|1)$} The goal of this subsection is to give the list of  character formulae of projective covers in the principal block of $\mc O$ of $\mf g:=\gl(2|1)$  computed in \cite{ChWa08}. With the canonical isomorphic $\gl(1|2)\cong \gl(2|1)$ and the BGG reciprocity, the composition factors of Verma modules over $\gl(1|2)$ can be read off from character formulae in this subsection.

 We choose the Borel subalgebra $\mf b :=\bigoplus_{1\leq i\leq j\leq 3}\C E_{ij}$ and the standard Cartan subalgebra $\h^\ast:=\oplus_{i=1}^3\C E_{ii}$. Set $\rho\in \h^\ast$ to be the corresponding Weyl vector. For $\mu \in \mf h^\ast$, we let  $M^{\mf a}(\mu)$ denote Verma module over $\mf{gl}(2|1)$ with highest weight $\mu -\rho$. Also, we let $P^{\mf a}(\mu)$   denote the projective cover of the simple quotient of $M^{\mf a}(\mu)$. For any $a,b,c\in \C$, we set $(a|b,c):= a\vare_1+b\vare_2+c\vare_3$. We adopt the notation $P^{\mf a}(\la) = \sum_{\mu \in \h^\ast}(P^{\mf a}(\la): M^{\mf a}(\mu))M^{\mf a}(\mu)$ to record the Verma flag structure $P^{\mf a}(\la).$
 
	\begin{lem} \label{lemapp1} \cite[Section 9]{ChWa08}
		We have the following character formulae:
		\begin{align*}
		&(1). ~P^{\mf a}(0,0|0) = M^{\mf a}(0,0|0)+M^{\mf a}(0,1|-1)+M^{\mf a}(1,0|-1). \\
		&(2). ~P^{\mf a}(0,-1|1) = M^{\mf a}(0,-1|1)+M^{\mf a}(0,0|0)+M^{\mf a}(1,0|-1).\\
		&(3). ~P^{\mf a}(-1,0|1) = M^{\mf a}(-1,0|1)+M^{\mf a}(0,-1|1)+M^{\mf a}(00|0).\\
		&(4). ~P^{\mf a}(0,-k|k) = M^{\mf a}(0,-k|k)+M^{\mf a}(0,-(k-1)|(k-1)),~\text{for $k> 1$}.\\
		&(5). ~P^{\mf a}(-k,0|k) = M^{\mf a}(-k,0|k)+M^{\mf a}(0,-k|k) +M^{\mf a}(-(k-1),0|(k-1))+\\&M^{\mf a}(0,-(k-1)|(k-1)),~\text{for $k> 1$}.\\
		&(6). ~P^{\mf a}(k,0|-k) = M^{\mf a}(k,0|-k)+M^{\mf a}(k+1,0|-(k+1)),~\text{for $k\geq 1$}.\\
		&(7). ~P^{\mf a}(0,k|-k) = M^{\mf a}(0,k|-k)+M^{\mf a}(k,0|-k)+M^{\mf a}(0,k+1|-(k+1))+\\&M^{\mf a}(k+1,0|-(k+1)),~\text{for $k\geq 1$}.
		\end{align*}	 
	\end{lem}

\vspace{2mm}

\noindent
Chih-Whi Chen:~Department of Mathematics, National Central University, Zhongli District, Taoyuan City, Taiwan;
E-mail: {\tt cwchen@math.ncu.edu.tw}
\hspace{2cm}

%%%%%%%%%%%%%%%%%%%%%%%%%%%%%%%%%%%%%%%%%%%%%%%%%%%%%%%%%%%%%%%%%%%%%%%%%%%%%%%

%%%%%%%%%

\end{document}